\documentclass[reqno]{amsart}

\usepackage{a4wide}
\usepackage{mathrsfs}
\usepackage{mathtools}
\usepackage{tikz-cd}
\usepackage{amsmath}
\usepackage{amssymb}
\usepackage{bbm}
\numberwithin{equation}{section}
\usepackage[colorlinks,citecolor=green,linkcolor=red]{hyperref}
\usepackage{hyphenat}

\usepackage[latin1]{inputenc}

\newcommand{\B}{\mathbb{B}}
\newcommand{\N}{\mathbb{N}}
\newcommand{\R}{\mathbb{R}}

\newcommand{\restr}[1]{\lower3pt\hbox{\(|_{#1}\)}}

\newcommand{\nchi}{{\raise.3ex\hbox{\(\chi\)}}}
\newcommand{\1}{\mathbbm 1}
\newcommand{\fr}{\penalty-20\null\hfill\(\blacksquare\)}
\newcommand{\X}{{\rm X}}

\newcommand{\mm}{\mathfrak m}

\newtheorem{theorem}{Theorem}[section]
\newtheorem{corollary}[theorem]{Corollary}
\newtheorem{lemma}[theorem]{Lemma}
\newtheorem{proposition}[theorem]{Proposition}
\newtheorem{definition}[theorem]{Definition}
\newtheorem{example}[theorem]{Example}
\newtheorem{remark}[theorem]{Remark}

\title{Tensor products of measurable Banach bundles}
\author{Milica Cakovi\'{c}}
\address{Department of Mathematics and Informatics, Faculty of Sciences,
University of Novi Sad, Trg D.\ Obradovi\'{c}a 4, 21000 Novi Sad, Serbia}
\email{milica.lucic@dmi.uns.ac.rs}

\author{Danka Lu\v{c}i\'{c}}
\address{Department of Mathematics and Statistics,
P.O.\ Box 35 (MaD), FI-40014 University of Jyvaskyla}
\email{danka.d.lucic@jyu.fi}

\author{Enrico Pasqualetto}
\address{Department of Mathematics and Statistics,
P.O.\ Box 35 (MaD), FI-40014 University of Jyvaskyla}
\email{enrico.e.pasqualetto@jyu.fi}

\begin{document}
\date{\today} 
\keywords{Measurable Banach bundle; \(L^\infty\)-Banach \(L^\infty\)-module; injective tensor product;
projective tensor product; von Neumann lifting}
\subjclass[2020]{46M05, 47A80, 18F15, 53C23, 28A51, 46G15}
\begin{abstract}
We study injective and projective tensor products of measurable Banach bundles.
More precisely, given two separable measurable Banach bundles ${\bf E}$, ${\bf F}$ defined over a probability
space $({\rm X},\Sigma,\mathfrak m)$, we construct two measurable Banach bundles ${\bf E}\hat\otimes_\varepsilon{\bf F}$
and ${\bf E}\hat\otimes_\pi{\bf F}$ over $({\rm X},\Sigma,\mathfrak m)$ such that
$\Gamma({\bf E}\hat\otimes_\varepsilon{\bf F})\cong\Gamma({\bf E})\hat\otimes_\varepsilon\Gamma({\bf F})$
and $\Gamma({\bf E}\hat\otimes_\pi{\bf F})\cong\Gamma({\bf E})\hat\otimes_\pi\Gamma({\bf F})$,
where ${\bf G}\mapsto\Gamma({\bf G})$ is the map assigning to a measurable Banach bundle ${\bf G}$ its space
of $L^\infty(\mathfrak m)$-sections, while $\Gamma({\bf E})\hat\otimes_\varepsilon\Gamma({\bf F})$ and
$\Gamma({\bf E})\hat\otimes_\pi\Gamma({\bf F})$ denote the injective and projective tensor products,
respectively, of $\Gamma({\bf E})$ and $\Gamma({\bf F})$ in the sense of $L^\infty(\mathfrak m)$-Banach $L^\infty(\mathfrak m)$-modules.
In combination with previous results, this provides a fiberwise representation of the injective tensor product
$\mathscr M\hat\otimes_\varepsilon\mathscr N$ and the projective tensor product $\mathscr M\hat\otimes_\pi\mathscr N$
of two countably-generated $L^\infty(\mathfrak m)$-Banach $L^\infty(\mathfrak m)$-modules $\mathscr M$, $\mathscr N$.
\end{abstract}
\maketitle
\tableofcontents
\section{Introduction}
The goal of this paper is to study tensor products of measurable Banach bundles and of their section spaces.
In accordance with \cite[Definition 4.1]{DMLP25}, by a \emph{separable Banach \(\mathbb U\)-bundle} we mean
a weakly measurable multivalued map \({\bf E}\colon\X\twoheadrightarrow\mathbb U\) whose values are closed
vector subspaces of \(\mathbb U\). Here, \((\X,\Sigma,\mm)\) is a given \(\sigma\)-finite measure space, while
the ambient target space \(\mathbb U\) is some fixed separable Banach space containing isometric copies of
every separable Banach space. We denote by \(\Gamma({\bf E})\) the collection of all bounded measurable sections
of \({\bf E}\), quotiented up to \(\mm\)-a.e.\ equality. Besides its Banach space structure with respect to
the supremum norm, the section space \(\Gamma({\bf E})\) inherits some additional operations from the pointwise
nature of its elements:
\begin{itemize}
\item A pointwise-norm operator \(|\cdot|\colon\Gamma({\bf E})\to L^\infty(\mm)^+\), which assigns to each element
\(v\in\Gamma({\bf E})\) the \(\mm\)-a.e.\ defined function \(\X\ni x\mapsto|v|(x)\coloneqq\|v(x)\|_{{\bf E}(x)}\).
\item A pointwise multiplication operation \(L^\infty(\mm)\times\Gamma({\bf E})\ni(f,v)\mapsto f\cdot v\in\Gamma({\bf E})\),
which is given by \((f\cdot v)(x)\coloneqq f(x)v(x)\) for \(\mm\)-a.e.\ \(x\in\X\).
\end{itemize}
In other words, the space \(\Gamma({\bf E})\) is a (countably-generated) \emph{\(L^\infty(\mm)\)-Banach \(L^\infty(\mm)\)-module},
in the sense of \cite{Gigli14}. Remarkably, the converse holds as well, as it was proved in \cite{DMLP25} (after
\cite{LP18}): every countably-generated \(L^\infty(\mm)\)-Banach \(L^\infty(\mm)\)-module \(\mathscr M\) is isomorphic
to the section space \(\Gamma({\bf E})\) of some separable Banach \(\mathbb U\)-bundle \({\bf E}\). Furthermore, the functor
\({\bf E}\mapsto\Gamma({\bf E})\) is an equivalence of categories, in the spirit of the classical Serre--Swan theorem \cite{Serre55,Swan62}.
It is also worth pointing out that if \({\bf E}\) is a constant bundle (i.e.\ for a separable Banach space \(\B\) we have
that \({\bf E}(x)\cong\B\) for every \(x\in\X\)), then \(\Gamma({\bf E})\) coincides with the \emph{Lebesgue--Bochner space}
\(L^\infty(\mm;\B)\).
\medskip

Let us describe the contents of this paper. Fix two separable Banach \(\mathbb U\)-bundles \({\bf E}\) and \({\bf F}\).
Thanks to the results of \cite{Pas23}, it is possible to construct different kinds of tensor products of their section
spaces \(\Gamma({\bf E})\) and \(\Gamma({\bf F})\). We shall focus on the \emph{injective tensor product}
\(\Gamma({\bf E})\hat\otimes_\varepsilon\Gamma({\bf F})\) and the \emph{projective tensor product}
\(\Gamma({\bf E})\hat\otimes_\pi\Gamma({\bf F})\). We highlight that here we are considering tensor products
in the sense of \(L^\infty(\mm)\)-Banach \(L^\infty(\mm)\)-modules, which is essential in order to capture
the pointwise behaviour of the sections of \({\bf E}\) and \({\bf F}\); the classical tensor products of Banach spaces
do not serve this purpose, as we will discuss in Examples \ref{ex:diff_inj_prod} and \ref{ex:diff_proj_prod}.
Since \(\Gamma({\bf E})\) and \(\Gamma({\bf F})\) are countably generated, it can be readily checked that also
\(\Gamma({\bf E})\hat\otimes_\varepsilon\Gamma({\bf F})\) and \(\Gamma({\bf E})\hat\otimes_\pi\Gamma({\bf F})\)
are countably generated. It follows that there exist (essentially) uniquely-determined separable Banach \(\mathbb U\)-bundles
\({\bf E}\hat\otimes_\varepsilon{\bf F}\) and \({\bf E}\hat\otimes_\pi{\bf F}\) on \((\X,\Sigma,\mm)\) such that
\[
\Gamma({\bf E}\hat\otimes_\varepsilon{\bf F})\cong\Gamma({\bf E})\hat\otimes_\varepsilon\Gamma({\bf F}),
\qquad\Gamma({\bf E}\hat\otimes_\pi{\bf F})\cong\Gamma({\bf E})\hat\otimes_\pi\Gamma({\bf F}).
\]
Our main goal is to characterise \({\bf E}\hat\otimes_\varepsilon{\bf F}\) and \({\bf E}\hat\otimes_\pi{\bf F}\).
In this respect, our results are the following.
\begin{itemize}
\item For injective tensor products, the best-case scenario occurs: as we prove in Section \ref{sec:inj_mod},
\[
({\bf E}\hat\otimes_\varepsilon{\bf F})(x)\cong{\bf E}(x)\hat\otimes_\varepsilon{\bf F}(x)\quad\text{ for }\mm\text{-a.e.\ }x\in\X,
\]
where \({\bf E}(x)\hat\otimes_\varepsilon{\bf F}(x)\) denotes the injective tensor product of Banach spaces.
Roughly speaking, this nice feature is ultimately due to the fact that injective tensor products of Banach spaces
respect subspaces; cf.\ Remark \ref{rmk:inj_tensor_products_subspaces}.
\item For projective tensor products, the situation is much more delicate. One of the reasons is that projective tensor
products (typically) do not respect subspaces; cf.\ Remark \ref{rmk:proj_tensor_products_subspaces}. To provide a description
of the fibers of \({\bf E}\hat\otimes_\pi{\bf F}\), we leverage the machinery of \emph{liftings} of \(L^\infty(\mm)\)-Banach
\(L^\infty(\mm)\)-modules, which we will discuss more in detail below. For the moment, we only mention that the
lifting \(\ell\mathscr M\) of an \(L^\infty(\mm)\)-Banach \(L^\infty(\mm)\)-module \(\mathscr M\) allows us to
select an everywhere-defined representative \(\ell v\in\ell\mathscr M\) of any element \(v\in\mathscr M\),
whose pointwise value \(\ell v_x\) at \(x\in\X\) belongs to a Banach space \(\ell\mathscr M_x\) called the \emph{fiber}
of \(\ell\mathscr M\) at \(x\). Shortly said, we show that for \(\mm\)-a.e.\ point \(x\in\X\) the fibers \({\bf E}(x)\)
and \({\bf F}(x)\) can be identified with closed subspaces of \(\ell{\bf E}_x\coloneqq\ell\Gamma({\bf E})_x\) and
\(\ell{\bf F}_x\coloneqq\ell\Gamma({\bf F})_x\), respectively; cf.\ Lemma \ref{lem:fibers_bundle_embed_lift}.
Under this identification and with some efforts, one can also prove that
\[
({\bf E}\hat\otimes_\pi{\bf F})(x)\cong{\rm cl}_{\ell{\bf E}_x\hat\otimes_\pi\ell{\bf F}_x}({\bf E}(x)\otimes{\bf F}(x))
\quad\text{ for }\mm\text{-a.e.\ }x\in\X;
\]
cf.\ Remark \ref{rmk:fibers_proj_bundle} and Theorem \ref{thm:proj_mod}. However, since \({\bf E}(x)\) and
\(\ell{\bf E}_x\) may not coincide (see Remark \ref{rmk:E_and_ellE} and Proposition \ref{prop:psi_x_not_surj}), it is not clear
whether \(({\bf E}\hat\otimes_\pi{\bf F})(x)\cong{\bf E}(x)\hat\otimes_\pi{\bf F}(x)\) holds for \(\mm\)-a.e.\ \(x\in\X\)
when \({\bf E}\) and \({\bf F}\) are generic Banach bundles. Sufficient conditions for this to hold are that
\({\bf E}(x)\) and \({\bf F}(x)\) are Hilbert for \(\mm\)-a.e.\ \(x\in\X\) (see Theorem \ref{thm:suff_ident_proj_fibers}),
or that \({\bf E}\) and \({\bf F}\) are constant bundles (see Lemma \ref{lem:proj_prod_const_bundles}).
\end{itemize}
The above-mentioned lifting theory for \(L^\infty(\mm)\)-Banach \(L^\infty(\mm)\)-modules, which was introduced by the
second and third named authors together with Di Marino in \cite{DMLP25}, relies on the existence of a von Neumann lifting
\(\ell\colon L^\infty(\mm)\to\mathcal L^\infty(\Sigma)\) (where \(\mathcal L^\infty(\Sigma)\) denotes the Banach algebra
of bounded measurable functions on \(\X\)) of the reference measure \(\mm\), and as such it requires a rather strong form
of the Axiom of Choice. The purpose of module liftings is to supply a consistent selection of representatives. Namely,
given any \(L^\infty(\mm)\)-Banach \(L^\infty(\mm)\)-module \(\mathscr M\), there exist a canonical
\(\mathcal L^\infty(\Sigma)\)-Banach \(\mathcal L^\infty(\Sigma)\)-module \(\ell\mathscr M\) and a linear embedding map
\(\ell\colon\mathscr M\to\ell\mathscr M\) satisfying
\[
|\ell v|=\ell|v|,\qquad\ell(f\cdot v)=(\ell f)\cdot(\ell v)\quad\text{ for every }v\in\mathscr M\text{ and }f\in L^\infty(\mm).
\]
It is well known that \(L^\infty(\mm)\) is the right functional space for the von Neumann theory of liftings; this
is the reason why we consider \(L^\infty(\mm)\)-Banach \(L^\infty(\mm)\)-modules instead of more well-established variants,
as \(L^p(\mm)\)-Banach \(L^\infty(\mm)\)-modules for \(p\in(1,\infty)\) or \(L^0(\mm)\)-Banach \(L^0(\mm)\)-modules.
It is also worth pointing out that, in order to extend the lifting theory to the section spaces \(\Gamma({\bf E})\)
of separable Banach \(\mathbb U\)-bundles \({\bf E}\), it is truly essential to leverage the technology of
\(L^\infty(\mm)\)-Banach \(L^\infty(\mm)\)-modules, because -- in general -- it is not possible to construct
liftings directly at the level of Banach bundles (roughly speaking, because the fibers \(\{\ell{\bf E}_x:x\in\X\}\)
of \(\ell\Gamma({\bf E})\) cannot be embedded into a common Banach space); cf.\ the discussion in \cite[Appendix C]{LP23}.
\medskip

We conclude the introduction by illustrating how our results contribute to the existing literature.
In the context of continuous bundles of Banach spaces over locally compact Hausdorff spaces, results similar to ours
for injective and projective tensor products were obtained by Kitchen and Robbins in \cite{KR81} (after \cite{KR82}).
In this paper, we consider measurable bundles of (separable) Banach spaces instead, in the sense of \cite{DMLP25};
other related notions of measurable Banach bundles were studied e.g.\ by Gutman in \cite{Gutman93}, and by the second
and third named authors together with Gigli in \cite{GLP22}. The axiomatisation we have chosen here is due to its
strong connections with the theory of countably-generated \(L^\infty(\mm)\)-Banach \(L^\infty(\mm)\)-modules,
as we discussed above. Indeed, our motivation mostly comes from the analysis of metric measure spaces: Gigli introduced
the language of \(L^p(\mm)\)-normed \(L^\infty(\mm)\)-modules in \cite{Gigli14} (and \cite{Gigli17}) with the aim of
developing an effective vector calculus for (possibly infinite-dimensional) non-smooth structures. In this regard,
\(L^p(\mm)\)-normed \(L^\infty(\mm)\)-modules were used to define suitable spaces of \(p\)-integrable \emph{\(1\)-forms},
\emph{vector fields} and other differential objects. Several constructions in the category of \(L^p(\mm)\)-Banach
\(L^\infty(\mm)\)-modules -- such as \emph{duals}, \emph{pullbacks} and \emph{tensor products of Hilbertian modules}
-- have had a fundamental r\^{o}le in many recent results about the analytic and geometric properties of metric measure
spaces. Fiberwise representations of \(L^p(\mm)\)-Banach \(L^\infty(\mm)\)-modules were proved to be extremely useful
in providing more explicit characterisations of duals and pullbacks; see \cite{GLP22}. In a sense, the results of this
paper complement those of \cite{GLP22}, as they provide fiberwise representations of the various tensor products of
\(L^\infty(\mm)\)-Banach \(L^\infty(\mm)\)-modules introduced by the third named author in \cite{Pas23}.

Let us also mention that, prior to the introduction of \(L^p(\mm)\)-normed \(L^\infty(\mm)\)-modules on metric measure
spaces, some strictly related concepts were already well established in the literature. For example, Guo developed in
\cite{Guo89,Guo92} the language of \emph{random normed modules} (see also \cite{Guo99}), after the work of Schweizer
and Sklar \cite{Schweizer} on probabilistic metric spaces. This theory, which has been investigated thoroughly by Guo
and his coauthors in a long series of works (see e.g.\ the survey \cite{Guo11} and the references therein, as well as \cite{GuoMuTu} or \cite{GuoWaXuYuCh} for more recent developments),
has applications in finance optimisation problems and in the modelling of conditional risk measures \cite{FilKuVo}.
Furthermore, Haydon, Levy and Raynaud applied the machinery of \emph{randomly normed \(L^0\)-modules} in \cite{HLR91}
to investigate ultraproducts of Lebesgue--Bochner spaces. Since random normed modules are -- in view of their applications
-- defined over finite (or \(\sigma\)-finite) measure spaces, in this paper we consider \(L^\infty(\mm)\)-Banach
\(L^\infty(\mm)\)-modules over \(\sigma\)-finite measure spaces (not just over metric measure spaces).
\medskip

\noindent\textbf{Acknowledgements.}
M.C.\ gratefully acknowledges the financial support of the Ministry of Science, Technological Development and Innovation of the Republic of Serbia
(Grants No.\ 451-03-137/2025-03/200125 \textup{\&} 451-03-136/2025-03/200125). D.L.\ was supported by the Research Council of Finland grant 362689.
E.P.\ was supported by the Research Council of Finland grant 362898.
\section{Preliminaries}
Let us fix some general notations and conventions, which will be used throughout the whole paper. Given a non-empty index set \(\mathcal I\)
and a family \((\X_i)_{i\in\mathcal I}\) of non-empty sets, we shall sometimes regard an element \(x_\star=(x_i)_{i\in\mathcal I}\) of
the Cartesian product \(\prod_{i\in\mathcal I}\X_i\) as a map \(x\colon\mathcal I\to\bigsqcup_{i\in\mathcal I}\X_i\) satisfying \(x(i)\in\X_i\)
for every \(i\in\mathcal I\). Moreover, we write \(O=O[A,B,\ldots]\) when we want to specify that an object \(O\) depends on some given data \(A,B,\ldots\)
\subsection{Measure theory}
Let us recall some notions and terminology in measure theory; cf.\ \cite{Bogachev07,Fremlin3}.
\subsubsection*{Lebesgue spaces}
Let \((\X,\Sigma,\mathcal N)\) be an \textbf{enhanced measurable space}, i.e.\ \((\X,\Sigma)\) is a measurable space
and \(\mathcal N\subseteq\Sigma\) is a \(\sigma\)-ideal. Letting \(\mathcal L^\infty(\Sigma)\) be the set of all
bounded measurable functions from \((\X,\Sigma)\) to \(\R\), we introduce the following equivalence relation
\(\sim_{\mathcal N}\) on \(\mathcal L^\infty(\Sigma)\): given \(f,g\in\mathcal L^\infty(\Sigma)\), we declare that
\(f\sim_{\mathcal N}g\) if and only if \(\{f\neq g\}\coloneqq\{x\in\X:f(x)\neq g(x)\}\in\mathcal N\). Then we define
\[
L^\infty(\mathcal N)\coloneqq\mathcal L^\infty(\Sigma)/\sim_{\mathcal N}.
\]
The space \(L^\infty(\mathcal N)\) is a unital, associative, commutative algebra with respect to the usual pointwise
operations. In addition, \(L^\infty(\mathcal N)\) is a Riesz space if endowed with the following partial order: given any
\(f,g\in L^\infty(\mathcal N)\), we declare that \(f\leq g\) if and only if \(f\leq g\) holds \textbf{\(\mathcal N\)-a.e.},
which means that \(\{\bar f>\bar g\}\in\mathcal N\) for some (thus, any) choice of representatives
\(\bar f,\bar g\in\mathcal L^\infty(\Sigma)\) of \(f\) and \(g\), respectively.
The positive cone of the Riesz space \(L^\infty(\mathcal N)\) will be denoted by \(L^\infty(\mathcal N)^+\coloneqq\{f\in L^\infty(\mathcal N):f\geq 0\}\).
Finally, the space \(L^\infty(\mathcal N)\) is a Banach space provided it is endowed with the following norm:
\[
\|f\|_{L^\infty(\mathcal N)}\coloneqq\inf_{\bar f\in\pi_{\mathcal N}^{-1}(f)}\sup_{x\in\X}|\bar f(x)|\quad\text{ for every }f\in L^\infty(\mathcal N),
\]
where \(\pi_{\mathcal N}\colon\mathcal L^\infty(\Sigma)\to L^\infty(\mathcal N)\) denotes the canonical projection map. Two important examples:
\begin{itemize}
\item If \(\mathcal N=\{\varnothing\}\), then \(L^\infty(\mathcal N)=\mathcal L^\infty(\Sigma)\)
and \(\|f\|_{\mathcal L^\infty(\Sigma)}=\sup_{x\in\X}|f(x)|\) for every \(f\in\mathcal L^\infty(\Sigma)\).
\item If \(\mm\geq 0\) is a given measure on \((\X,\Sigma)\), then \(\mathcal N_\mm\coloneqq\{N\in\Sigma:\mm(N)=0\}\) is a \(\sigma\)-ideal
and \(L^\infty(\mm)\coloneqq L^\infty(\mathcal N_\mm)\) is the usual \(\infty\)-Lebesgue space over the measure space \((\X,\Sigma,\mm)\).
\end{itemize}
For any set \(E\in\Sigma\), we denote by \(\1_E\in\mathcal L^\infty(\Sigma)\) its characteristic function and by \(\1_E^\mm\in L^\infty(\mm)\)
the \(\mm\)-a.e.\ equivalence class of \(\1_E\).
\begin{remark}\label{rmk:wlog_m_fin_complete}{\rm
For any \(\sigma\)-finite measure space \((\X,\Sigma,\mm)\), one can easily construct a probability measure \(\mu\) on \((\X,\Sigma)\)
such that \(\mm\ll\mu\ll\mm\). Letting \((\X,\bar\Sigma,\bar\mu)\) be the completion of the measure space \((\X,\Sigma,\mu)\), we then
have that \(L^\infty(\bar\mu)\) can be canonically identified with \(L^\infty(\mm)\). Due to this reason, in the sequel we shall often
assume without loss of generality that the measure spaces under consideration are complete probability spaces.
\fr}\end{remark}
If \((\X,\Sigma,\mm)\) is a \(\sigma\)-finite measure space, then for any order-bounded subset \(\{f_i\}_{i\in\mathcal I}\) of \(L^\infty(\mm)\)
(meaning that there exists \(g\in L^\infty(\mm)^+\) such that \(|f_i|\leq g\) for every \(i\in\mathcal I\)) there exists a countable family
\(\mathcal C\subseteq\mathcal I\) such that \(\sup_{i\in\mathcal C}f_i=\bigvee_{i\in\mathcal I}f_i\) and \(\inf_{i\in\mathcal C}f_i=\bigwedge_{i\in\mathcal I}f_i\).
In other words, the Riesz space \(L^\infty(\mm)\) is Dedekind complete and it has the countable supremum and infimum properties.
\subsubsection*{Liftings of measures}
Let \((\X,\Sigma,\mm)\) be a complete probability space. The von Neumann--Maharam lifting theorem \cite{Fremlin3} ensures the existence
of a linear map \(\ell\colon L^\infty(\mm)\to\mathcal L^\infty(\Sigma)\) satisfying
\[\begin{split}
\ell(fg)=(\ell f)(\ell g)&\quad\text{ for every }f,g\in L^\infty(\mm),\\
\|\ell f\|_{\mathcal L^\infty(\Sigma)}=\|f\|_{L^\infty(\mm)}&\quad\text{ for every }f\in L^\infty(\mm),\\
\pi_\mm(\ell f)=f&\quad\text{ for every }f\in L^\infty(\mm),\\
\ell(c\1_\X^\mm)=c\1_\X &\quad\text{ for every }c\in\R,
\end{split}\]
where we set \(\pi_\mm\coloneqq\pi_{\mathcal N_\mm}\colon\mathcal L^\infty(\Sigma)\to L^\infty(\mm)\).  
We say that \(\ell\) is a \textbf{lifting} of \(\mm\). If \(f,g\in L^\infty(\mm)\) satisfy \(f\leq g\), then \(\ell f\leq\ell g\).
Given any \(E\in\Sigma\), there exists a (unique) set \(\ell E\in\Sigma\) such that
\[
\ell\1_E^\mm=\1_{\ell E}.
\]
The resulting map \(\ell\colon\Sigma\to\Sigma\)
is a Boolean homomorphism such that
\[\begin{split}
\mm(E\Delta\ell E)=0&\quad\text{ for every }E\in\Sigma,\\
\ell N=\varnothing&\quad\text{ for every }N\in\mathcal N_\mm.
\end{split}\]
\subsubsection*{Atoms}
Let \((\X,\Sigma,\mm)\) be a complete probability space. A given set \(A\in\Sigma\) is said to be an \textbf{atom} of the measure \(\mm\) provided it holds that \(\mm(A)>0\) and 
\[
B\in\Sigma,\;B\subseteq A\quad\Longrightarrow\quad\text{either }\mm(B)=0\text{ or }\mm(A\setminus B)=0.
\]
Given any \(A,\tilde A\in \Sigma\) with \(\mm(A\Delta \tilde A)=0\), it holds that \(A\) is an atom of \(\mm\) if and only if \(\tilde A\) is an atom of \(\mm\),
thus there are at most countably many atoms of \(\mm\) that are not mutually \(\mm\)-a.e.\ equivalent.
\medskip

Assuming \(\ell\) is a lifting of \(\mm\), we denote by \(\mathscr A(\ell)\) the family of all \textbf{lifted atoms} of \(\ell\), i.e.\ we set
\[
\mathscr A(\ell)\coloneqq\{\ell A\;|\;A\in\Sigma\text{ is an atom of }\mm\}\subseteq\Sigma\setminus\mathcal N_\mm.
\]
The above-discussed properties of atoms ensure that the family \(\mathscr A(\ell)\) is at most countable, so that
\[
{\rm A}(\ell)\coloneqq\bigcup_{A\in\mathscr A(\ell)}A\in\Sigma.
\]
Moreover, it can be readily checked that \(\mathscr A(\ell)\) coincides with the collection of all minimal elements of the
poset \((\{\ell E:E\in\Sigma\setminus\mathcal N_\mm\},\subseteq)\). The following result was proved in \cite[Lemma 4.7]{GLP22}.
\begin{lemma}\label{lem:lifted_atoms}
Let \((\X,\Sigma,\mm)\) be a complete probability space. Let \(\ell\) be a lifting of \(\mm\). Then:
\begin{itemize}
\item[{\rm i)}] \(\{x\}\in\Sigma\) for every \(x\in\X\setminus{\rm A}(\ell)\). 
\item[{\rm ii)}] Given any \(E\in\Sigma\) and \(A\in\mathscr A(\ell)\), we have that either \(A\subseteq\ell E\) or \(A\cap\ell E=\varnothing\).
\item[{\rm iii)}] Given any \(f\in L^\infty(\mm)\) and \(A\in\mathscr A(\ell)\), the function \(\ell f\) is constant on \(A\).
\end{itemize}
\end{lemma}
As a consequence of the above lemma, we have that the elements of \(\mathscr A(\ell)\) are pairwise disjoint.
The proof of the next result is essentially contained in the one of \cite[Lemma C.1]{LP23}.
\begin{lemma}\label{lem:point_missing_union_liftings}
Let \((\X,\Sigma,\mm)\) be a complete probability space. Assume \(\mm\) is atomless, i.e.\ it has no atoms.
Let \(x\in\X\) be given. Then there exists a partition \((E_n)_{n\in\N}\subseteq\Sigma\setminus \mathscr N_\mm\) of \(\X\) such that
\[
x\notin\bigcup_{n\in\N}\ell E_n.
\]
\end{lemma}
\begin{proof}
Given that \(\mm\) is atomless, for any \(E\in\Sigma\) and \(\lambda\in(0,\mm(E))\) there exists a set \(F\in\Sigma\)
such that \(F\subseteq E\) and \(\mm(F)=\lambda\); see e.g.\ \cite[Theorem B.3]{Bon:Pas:Raj:20}. By repeatedly making
use of this fact, we shall construct the sought sequence \((E_n)_{n\in\N}\) in a recursive manner. First of all, fix
any set \(E_1\in\Sigma\) such that \(\mm(E_1)=\frac{1}{2}\). Since \(\ell E_1\cup\ell(\X\setminus E_1)=\ell\X=\X\) and
\(\ell E_1\cap\ell(\X\setminus E_1)=\ell\varnothing=\varnothing\), we can assume (up to replacing \(E_1\) with
\(\X\setminus E_1\)) that \(x\notin\ell E_1\). Now, assume that we have already chosen pairwise disjoint sets
\(E_1,\ldots,E_n\in\Sigma\), for some \(n\in\N\) with \(n\geq 1\), in such a way that \(\mm(E_i)=2^{-i}\) for every
\(i=1,\ldots,n\) and \(x\notin\ell E_1\cup\ldots\cup\ell E_n\). Since
\(\mm\big(\X\setminus\bigcup_{i=1}^n E_i\big)=2^{-n}\), by arguing as above we can find a subset \(E_{n+1}\in\Sigma\)
of \(\X\setminus\bigcup_{i=1}^n E_i\), such that \(\mm(E_{n+1})=2^{-(n+1)}\) and \(x\notin \ell E_{n+1}\).
This way, we obtain a sequence \((E_n)_{n=1}^\infty\subseteq\Sigma\) of pairwise disjoint sets satisfying \(x\notin\bigcup_{n=1}^\infty\ell E_n\)
and \(\mm\big(\bigcup_{n=1}^\infty E_n\big)=\sum_{n=1}^\infty 2^{-n}=1\). In particular, we have that
\(E_0\coloneqq\X\setminus\bigcup_{n=1}^\infty E_n\) is \(\mm\)-null, so that \(\ell E_0=\varnothing\). Therefore,
\((E_n)_{n\in\N}\) is a partition of \(\X\) such that \(x\notin\bigcup_{n\in\N}\ell E_n\), thus accordingly the
statement is achieved.
\end{proof}
\subsection{Algebraic tensor products of modules}
Let us recall the notion of (algebraic) tensor product of modules \cite{Bourbaki48} (see also \cite{Conrad18}).
Let \(R\) be a commutative ring with a multiplicative identity. Let \(M\) and \(N\) be \(R\)-modules. Then there
exists a unique couple \((M\otimes N,\otimes)\), where \(M\otimes N\) is an \(R\)-module (called the \textbf{algebraic
tensor product} of \(M\) and \(N\)) and \(\otimes\colon M\times N\to M\otimes N\) is an \(R\)-bilinear map, satisfying
the following universal property: given any \(R\)-module \(Q\) and any \(R\)-bilinear map \(b\colon M\times N\to Q\),
there exists a unique \(R\)-linear map \(\tilde b\colon M\otimes N\to Q\) such that
\[\begin{tikzcd}
M\times N \arrow[d,swap,"\otimes"] \arrow[r," b"] & Q \\
M\otimes N \arrow[ur,swap,"\tilde b"]
\end{tikzcd}\]
is a commutative diagram. We say that \(\tilde b\) is the \textbf{\(R\)-linearisation} of \(b\).
The pair \((M\otimes N,\otimes)\) is unique up to a unique isomorphism: given any pair \((W,\tilde\otimes)\) having the same
properties, there exists a unique isomorphism of \(R\)-modules \(\Phi\colon M\otimes N\to W\) such that the diagram
\[\begin{tikzcd}
M\times N \arrow[r,"\otimes"] \arrow[dr,swap," \tilde\otimes"] & M\otimes N \arrow[d,"\Phi"]\\
& W
\end{tikzcd}\]
commutes. The elements of \(M\otimes N\) are called \textbf{tensors}. Those tensors of the form \(v\otimes w\)
(i.e.\ the elements of the image of \(\otimes\colon M\times N\to M\otimes N\)) are called \textbf{elementary tensors}.
Each \(\alpha\in M\otimes N\) is a sum of elementary tensors, i.e.\ \(\alpha=\sum_{i=1}^n v_i\otimes w_i\)
for some \(v_1,\ldots,v_n\in M\) and \(w_1,\ldots,w_n\in N\).
In the particular case where \(R\) is the real field \(\R\) (so that the \(R\)-modules are the real vector spaces),
with the above construction one recovers the algebraic tensor product of real vector spaces.
\medskip

Let us recall a useful criterion to detect null tensors: given any \(\alpha=\sum_{i=1}^n v_i\otimes w_i\in M\otimes N\),
we have that \(\alpha=0\) if and only if
\begin{equation}\label{eq:alg_criterion_null_tensor}
\sum_{i=1}^n b(v_i,w_i)=0\quad\text{ for every }R\text{-module }Q\text{ and every }R\text{-bilinear map }b\colon M\times N\to Q.
\end{equation}
\subsection{Banach spaces}
Let \(\B\) be a Banach space. We denote by \(\mathbb D_\B\) the unit disc of \(\B\), i.e.\ we set
\[
\mathbb D_\B\coloneqq\{v\in\B:\|v\|_\B\leq 1\}.
\]
We denote by \(\B'\) its dual Banach space and by \(\langle\cdot,\cdot\rangle\colon\B'\times\B\to\R\) the duality pairing
\[
\B'\times\B\ni(\omega,v)\mapsto\langle\omega,v\rangle\coloneqq\omega(v)\in\R.
\]
Given another Banach space \(\mathbb V\), we say that \(\B\) and \(\mathbb V\) are \textbf{isomorphic} (as Banach spaces),
and we write \(\B\cong\mathbb V\), provided there exists a linear isometric bijection \(\Phi\colon\B\to\mathbb V\). In this case,
many authors would write `isometrically isomorphic' instead.
\medskip

Given any vector subspace \(V\) of \(\B'\), it holds that
\begin{equation}\label{eq:equiv_w*_dense}
V\text{ is weakly\(^*\) dense in }\B'\quad\Longleftrightarrow\quad V\text{ separates the points of }\B.
\end{equation}
See e.g.\ \cite[Corollary 5.108]{AliprantisBorder99} for a proof of the above fact. Later, we will need the following standard result,
which we also prove for the reader's usefulness.
\begin{lemma}\label{lem:separating_points}
Let \(\B\) be a Banach space. Let \(\{v_i\}_{i\in I}\subseteq\B\) be dense in \(\B\). Assume \(\{\omega_i\}_{i\in I}\subseteq\B'\)
satisfies \(\|\omega_i\|_{\B'}=1\) and \(\langle\omega_i,v_i\rangle=\|v_i\|_\B\) for every \(i\in I\). Then \(\{\omega_i\}_{i\in I}\)
separates the points of \(\B\). In particular, the linear span of \(\{\omega_i\}_{i\in I}\) is weakly\(^*\) dense in \(\B'\).
\end{lemma}
\begin{proof}
We aim to show that if \(v\in\B\), \(\|v\|_\B=1\) satisfies \(\langle\omega_i,v\rangle=0\) for every \(i\in I\), then \(v=0\).
We argue by contradiction: assume \(v\neq 0\). Then we can find \(i_0\in I\) such that \(\|v_{i_0}-v\|_\B<\frac{\|v\|_\B}{2}\). Then
\[
\frac{1}{2}=\frac{1}{\|v\|_\B}\bigg(\|v\|_\B-\frac{\|v\|_\B}{2}\bigg)<\frac{\|v_{i_0}\|_\B}{\|v\|_\B}
=\frac{\langle\omega_{i_0},v_{i_0}\rangle}{\|v\|_\B}=\frac{\langle\omega_{i_0},v_{i_0}-v\rangle}{\|v\|_\B}
\leq\frac{\|\omega_{i_0}\|_{\B'}\|v_{i_0}-v\|_\B}{\|v\|_\B}<\frac{1}{2},
\]
which leads to a contradiction. Therefore, we have shown that \(\{\omega_i\}_{i\in I}\) separates the points of \(\B\).
The last part of the statement then follows from \eqref{eq:equiv_w*_dense}.
\end{proof}

Given a vector space \(\X\) and a seminorm \(p\) on \(\X\), it holds that \(\{p=0\}\coloneqq\{v\in\X:p(v)=0\}\) is a closed vector subspace of \(\X\).
Then the quotient \(\X/p\coloneqq\X/\{p=0\}\) is a vector space and
\[
\|v\|_{\X/p}\coloneqq p(\bar v)\quad\text{ for some (thus, any) representative }\bar v\in\X\text{ of }v\in\X/p
\]
defines a norm \(\|\cdot\|_{\X/p}\) on \(\X/p\). If the seminorm \(p\) is complete, then \(\X/p\) is a Banach space.
\subsubsection*{Tensor products of Banach spaces}
We fix some notations and terminology about tensor products of Banach spaces; cf.\ the monograph \cite{Ryan02} and the references
therein for more on this topic.
\medskip

Given Banach spaces \(\B\) and \(\mathbb V\), we can consider their algebraic tensor product \(\B\otimes\mathbb V\) as vector spaces.
The following result provides a criterion to detect when a given tensor \(\alpha\in\B\otimes\mathbb V\) is null.
\begin{lemma}[Null tensors on Banach spaces]\label{lem:null_tensor_Banach}
Let \(\B\), \(\mathbb V\) be Banach spaces. Let \(S_\B\) and \(S_{\mathbb V}\) be weakly\(^*\) dense subsets of \(\B'\)
and \(\mathbb V'\), respectively. Let \(\alpha=\sum_{i=1}^n v_i\otimes w_i\in\B\otimes\mathbb V\)
be given. Then it holds that \(\alpha=0\) if and only if
\[
\sum_{i=1}^n\langle\omega,v_i\rangle\langle\eta,w_i\rangle=0\quad\text{ for every }\omega\in S_\B\text{ and }\eta\in S_{\mathbb V}.
\]
\end{lemma}

Following \cite[Section 3.1]{Ryan02}, we denote by \(\B\otimes_\varepsilon\mathbb V\coloneqq(\B\otimes\mathbb V,\|\cdot\|_\varepsilon)\)
the normed space obtained by equipping the algebraic tensor product \(\B\otimes\mathbb V\) with the injective norm \(\|\cdot\|_\varepsilon\),
while \(\B\hat\otimes_\varepsilon\mathbb V\) denotes the \textbf{injective tensor product} of \(\B\) and \(\mathbb V\), which is defined
as the completion of \(\B\otimes_\varepsilon\mathbb V\).
\begin{remark}[Injective tensor products of subspaces]\label{rmk:inj_tensor_products_subspaces}{\rm
The injective tensor product respects subspaces, meaning that if \(\mathbb B\), \(\mathbb V\) are given Banach spaces
and \(\tilde\B\), \(\tilde{\mathbb V}\) are closed vector subspaces of \(\B\) and \(\mathbb V\), respectively, then the
norm of a tensor \(\alpha\in\tilde\B\otimes\tilde{\mathbb V}\) computed in \(\tilde\B\hat\otimes_\varepsilon\tilde{\mathbb V}\)
coincides with its norm in \(\B\hat\otimes_\varepsilon\mathbb V\). In particular, we have that
\(\tilde\B\hat\otimes_\varepsilon\tilde{\mathbb V}\cong{\rm cl}_{\B\hat\otimes_\varepsilon\mathbb V}(\B\otimes\mathbb V)\) as Banach spaces.
\fr}\end{remark}

Following \cite[Section 2.1]{Ryan02}, we denote by \(\B\otimes_\pi\mathbb V\coloneqq(\B\otimes\mathbb V,\|\cdot\|_\pi)\)
the normed space obtained by equipping the algebraic tensor product \(\B\otimes\mathbb V\) with the projective norm \(\|\cdot\|_\pi\),
while \(\B\hat\otimes_\pi\mathbb V\) denotes the \textbf{projective tensor product} of \(\B\) and \(\mathbb V\), which is defined
as the completion of \(\B\otimes_\pi\mathbb V\).
\begin{remark}[Projective tensor products of subspaces]\label{rmk:proj_tensor_products_subspaces}{\rm
Differently from injective tensor products,
the projective tensor product does not respect subspaces. Namely, if \(\mathbb B\), \(\mathbb V\) are given Banach spaces
and \(\tilde\B\), \(\tilde{\mathbb V}\) are closed vector subspaces of \(\B\) and \(\mathbb V\), respectively, then the
norm of a tensor \(\alpha\in\tilde\B\otimes\tilde{\mathbb V}\) computed in \(\tilde B\hat\otimes_\pi\tilde{\mathbb V}\)
can differ from its norm as an element of \(\B\hat\otimes_\pi\mathbb V\).
However, it is known (see \cite[Proposition 2.11]{Ryan02}) that \emph{\(\tilde\B\hat\otimes_\pi\tilde{\mathbb V}\)
is a subspace of \(\B\hat\otimes_\pi\mathbb V\) if and only if every bounded bilinear form on
\(\tilde\B\times\tilde{\mathbb V}\) can be extended to a bounded bilinear form on \(\B\times\mathbb V\) having the
same operator norm.} As a consequence, the following result holds (cf.\ \cite[Proposition 2.4]{Ryan02}):
\emph{if \(\tilde\B\), \(\tilde{\mathbb V}\) are \(1\)-complemented subspaces of \(\B\) and \(\mathbb V\), respectively,
then \(\tilde\B\hat\otimes_\pi\tilde{\mathbb V}\) is a subspace of \(\B\hat\otimes_\pi\mathbb V\).}
Let us recall that a \textbf{projection} \(T\colon\B\to\B\) is a bounded linear operator satisfying \(T\circ T=T\),
while \(\tilde\B\) is said to be a \textbf{\(1\)-complemented} subspace of \(\B\) provided \(\tilde\B=T(\B)\) for
some projection \(T\colon\B\to\B\) of operator norm \(1\). We also recall that every closed vector subspace of a
Hilbert space is \(1\)-complemented. In fact, Hilbert spaces are the only Banach spaces with the property that all
its subspaces are \(1\)-complemented; see e.g.\ \cite[Theorem 3.1]{Ran:01}.
\fr}\end{remark}
\begin{remark}[Tensor products of separable Banach spaces]\label{rmk:tensor_Banach_separable}{\rm
Let \(\B_1\), \(\B_2\) be given Banach spaces. Let \(\mathbb V_1\), \(\mathbb V_2\) be vector subspaces of
\(\B_1\) and \(\B_2\), respectively. Assume \(D_1\), \(D_2\) are dense \(\mathbb Q\)-vector subspaces
of \(\mathbb V_1\) and \(\mathbb V_2\), respectively. Then it can be readily checked that
\[\begin{split}
{\rm cl}_{\B_1\hat\otimes_\varepsilon\B_2}\bigg(\bigg\{\sum_{i=1}^n v_i\otimes w_i\;\bigg|\;n\in\N,\,(v_i)_{i=1}^n\subseteq D_1,\,(w_i)_{i=1}^n\subseteq D_2\bigg\}\bigg)
&={\rm cl}_{\B_1\hat\otimes_\varepsilon\B_2}(\mathbb V_1\otimes\mathbb V_2),\\
{\rm cl}_{\B_1\hat\otimes_\pi\B_2}\bigg(\bigg\{\sum_{i=1}^n v_i\otimes w_i\;\bigg|\;n\in\N,\,(v_i)_{i=1}^n\subseteq D_1,\,(w_i)_{i=1}^n\subseteq D_2\bigg\}\bigg)
&={\rm cl}_{\B_1\hat\otimes_\pi\B_2}(\mathbb V_1\otimes\mathbb V_2).
\end{split}\]
In particular, if \(\B_1\) and \(\B_2\) are separable, then \(\B_1\hat\otimes_\varepsilon\B_2\)
and \(\B_1\hat\otimes_\pi\B_2\) are separable.
\fr}\end{remark}
\subsubsection*{Bounded multilinear operators}
Let \(\B_1,\ldots,\B_k,\mathbb V\) be given Banach spaces, for some \(k\in\N\). Then we denote by
\({\rm M}_k(\B_1,\ldots,\B_k;\mathbb V)\) the space of all bounded multilinear operators
\(\varphi\colon\B_1\times\ldots\times\B_k\to\mathbb V\). We recall that \({\rm M}_k(\B_1,\ldots,\B_k;\mathbb V)\)
is a Banach space with respect to the operator norm
\[
\|\varphi\|_{{\rm M}_k(\B_1,\ldots,\B_k;\mathbb V)}\coloneqq\sup\big\{|\varphi(v_1,\ldots,v_k)|\;\big|\;
v_1\in\mathbb D_{\B_1},\ldots,v_k\in\mathbb D_{\B_k}\big\}.
\]
In the case \(k=2\), we write \({\rm B}(\B_1,\B_2;\mathbb V)\coloneqq{\rm M}_2(\B_1,\B_2;\mathbb V)\). We denote
\(\textsc{Hom}(\B_1;\mathbb V)\coloneqq{\rm M}_1(\B_1;\mathbb V)\), thus in particular \(\B'_1=\textsc{Hom}(\B_1;\R)\).
\subsection{Separable Banach bundles}
We recall that a Banach space \(\mathbb U\) is said to be a \textbf{universal separable Banach space} if
it is a separable Banach space where any separable Banach space can be embedded linearly and isometrically.
The Banach--Mazur theorem states that \(C([0,1])\) is a universal separable Banach space, cf.\ \cite{BP75}.
Following \cite[Definition 4.1]{DMLP25} (and \cite[Definition 2.15]{LP23}), by a \textbf{separable Banach
\(\mathbb U\)-bundle} over a given measurable space \((\X,\Sigma)\) we mean a multivalued map
\({\bf E}\colon\X\twoheadrightarrow\mathbb U\) having these two properties:
\begin{itemize}
\item \({\bf E}\) is \textbf{weakly measurable} in the sense of \cite[Definition 18.1]{AliprantisBorder99},
i.e.
\[
\{x\in\X\;|\;{\bf E}(x)\cap U\neq\varnothing\}\in\Sigma\quad\text{ for every open set }U\subseteq\mathbb U.
\]
\item \({\bf E}(x)\) is a closed vector subspace of \(\mathbb U\) for every \(x\in\X\).
\end{itemize}
The \textbf{section space} of \(\bf E\) is defined as the set \(\bar\Gamma({\bf E})\) of all
bounded measurable maps \(v\colon\X\to\mathbb U\) that satisfy \(v(x)\in{\bf E}(x)\) for every \(x\in\X\).
In particular, given any \(v\in\bar\Gamma({\bf E})\), we have that \(\X\ni x\mapsto\|v(x)\|_{{\bf E}(x)}\)
is measurable. We say that a set \(S\subseteq\bar\Gamma({\bf E})\) is \textbf{fiberwise dense} in \(\bar\Gamma({\bf E})\) if
\[
{\rm cl}_{\mathbb U}(\{v(x)\,|\,v\in S\})={\bf E}(x)\quad\text{ for every }x\in\X.
\]
Note that \(\bar\Gamma({\bf E})\) is an \(\mathcal L^\infty(\Sigma)\)-module with respect to the usual pointwise operations,
as well as a Banach space if equipped with the norm \(\|\cdot\|_{\bar\Gamma({\bf E})}\), which we define as
\[
\|v\|_{\bar\Gamma({\bf E})}\coloneqq\big\|\|v(\star)\|_{{\bf E}(\star)}\big\|_{\mathcal L^\infty(\Sigma)}
\quad\text{ for every }v\in\bar\Gamma({\bf E}).
\]
We report the following useful characterisation of separable Banach \(\mathbb U\)-bundles, whose proof can
be found in \cite[Proposition 4.4]{DMLP25} and relies upon the Kuratowski--Ryll-Nardzewski selection theorem.
\begin{proposition}\label{prop:select_dense}
Let \((\X,\Sigma)\) be a measurable space and \(\mathbb U\) a universal separable Banach space.
Then a multivalued map \({\bf E}\colon\X\twoheadrightarrow\mathbb U\) is a separable Banach \(\mathbb U\)-bundle
if and only if there exists a sequence \(\{v_n\}_{n\in\N}\) of measurable maps \(v_n\colon\X\to\mathbb U\) such that
\({\rm cl}_\B(\{v_n(x):n\in\N\})={\bf E}(x)\) holds for every \(x\in\X\). In this case, there
exists a countable \(\mathbb Q\)-vector subspace \(\mathcal C\) of \(\bar\Gamma({\bf E})\)
that is fiberwise dense in \(\bar\Gamma({\bf E})\).
\end{proposition}

Next, let us fix a measure \(\mm\) on \((\X,\Sigma)\). Then we define \(\Gamma({\bf E})\) as the quotient of
\(\bar\Gamma({\bf E})\) up to \(\mm\)-a.e.\ equality, i.e.\ we identify \(v,w\in\bar\Gamma({\bf E})\) if and
only if \(\|v(\star)-w(\star)\|_{{\bf E}(\star)}=0\) holds \(\mm\)-a.e.\ on \(\X\). We denote by \(\pi_\mm\colon\bar\Gamma({\bf E})\to\Gamma({\bf E})\)
the canonical projection map. Note that \(\Gamma({\bf E})\) is an \(L^\infty(\mm)\)-module with respect to the usual
\(\mm\)-a.e.\ pointwise operations, as well as a Banach space if equipped with the quotient norm \(\|\cdot\|_{\Gamma({\bf E})}\),
which is given by
\[
\|v\|_{\Gamma({\bf E})}\coloneqq\inf\big\{\|\bar v\|_{\bar\Gamma({\bf E})}\;\big|\;\bar v\in\bar\Gamma({\bf E}),
\,\pi_\mm(\bar v)=v\big\}\quad\text{ for every }v\in\Gamma({\bf E}).
\]
When \({\bf E}\) is a \textbf{constant bundle} (i.e.\ for some separable Banach space \(\B\)
we have that \({\bf E}(x)=\B\) for every \(x\in\X\)), the section space \(\Gamma({\bf E})\) coincides with
the \textbf{Lebesgue--Bochner space} \(L^\infty(\mm;\B)\).
\medskip

In the sequel, we shall need the following notion, which is taken from \cite[Definition 4.5]{DMLP25}.
\begin{definition}[Measurable collection of separable Banach spaces]
Let \((\X,\Sigma)\) be a measurable space. Then we say that \(\mathfrak E=\big(\{{\rm E}(x)\}_{x\in\X},\{v_n\}_{n\in\N},\{\omega_n\}_{n\in\N}\big)\)
is a \textbf{measurable collection of separable Banach spaces} over \((\X,\Sigma)\) provided the following conditions are satisfied:
\begin{itemize}
\item[\(\rm i)\)] \({\rm E}(x)\) is a separable Banach space for every \(x\in\X\).
\item[\(\rm ii)\)] \(\{v_n\}_{n\in\N}\subseteq\prod_{x\in\X}{\rm E}(x)\) satisfy \({\rm cl}_{{\rm E}(x)}(\{v_n(x):n\in\N\})={\rm E}(x)\) for every \(x\in\X\).
\item[\(\rm iii)\)] \(\{\omega_n\}_{n\in\N}\subseteq\prod_{x\in\X}{\rm E}(x)'\) satisfy
\[
\|\omega_n(x)\|_{{\rm E}(x)'}=\1_{\{v_n=0\}}(x),\qquad\langle\omega_n(x),v_n(x)\rangle=\|v_n(x)\|_{{\rm E}(x)}
\]
for every \(n\in\N\) and \(x\in\X\).
\item[\(\rm iv)\)] The function \(\X\ni x\mapsto\langle\omega_n(x),v_m(x)\rangle\in\R\) is measurable for every \(n,m\in\N\).
\end{itemize}
\end{definition}

Each measurable collection of separable Banach spaces induces a separable Banach \(\mathbb U\)-bundle, as it is shown by the following
result, which was proved in \cite[Theorem 4.6]{DMLP25}.
\begin{theorem}\label{thm:embed_meas_coll}
Let \((\X,\Sigma)\) be a measurable space and let \(\mathbb U\) be a universal separable Banach space.
Let \(\mathfrak E=\big(\{{\rm E}(x)\}_{x\in\X},\{v_n\}_{n\in\N},\{\omega_n\}_{n\in\N}\big)\) be a measurable collection of separable Banach spaces
over \((\X,\Sigma)\). Then there exists a family \({\rm I}_\star=\{{\rm I}_x\}_{x\in\X}\) of linear isometric embeddings
\({\rm I}_x\colon{\rm E}(x)\to\mathbb U\) such that the map \(\X\ni x\mapsto{\rm I}_x(v_n(x))\in\mathbb U\) is measurable
for every \(n\in\N\). We say that \({\rm I}_\star\) is a \textbf{measurable collection of linear isometric embeddings}
associated with \(\mathfrak E\). Moreover, the multivalued map \(\X\ni x\mapsto{\rm I}_x({\rm E}(x))\subseteq\mathbb U\)
is a separable Banach \(\mathbb U\)-bundle, which we call the \textbf{separable Banach \(\mathbb U\)-bundle induced by \({\rm I}_\star\)}.
\end{theorem}
\section{The theory of \texorpdfstring{\(L^\infty(\mathcal N)\)}{Linfty(N)}-Banach \texorpdfstring{\(L^\infty(\mathcal N)\)}{Linfty(N)}-modules}
Let us recall the concept of \(L^\infty(\mathcal N)\)-Banach \(L^\infty(\mathcal N)\)-module, which has been
introduced in \cite[Definition 4.3]{GLP22} as a generalisation of Gigli's notion of \(L^\infty(\mm)\)-Banach \(L^\infty(\mm)\)-module \cite{Gigli14}.
\begin{definition}[\(L^\infty(\mathcal N)\)-Banach \(L^\infty(\mathcal N)\)-module]
Let \((\X,\Sigma,\mathcal N)\) be an enhanced measurable space and \(\mathscr M\) an \(L^\infty(\mathcal N)\)-module.
Then we say that \(\mathscr M\) is an \textbf{\(L^\infty(\mathcal N)\)-normed \(L^\infty(\mathcal N)\)-module} if it is endowed with a map
\(|\cdot|\colon\mathscr M\to L^\infty(\mathcal N)^+\), called a \textbf{pointwise norm} on \(\mathscr M\), such that:
\begin{itemize}
\item[\(\rm i)\)] \(|v|\geq 0\) for every \(v\in\mathscr M\), with equality if and only if \(v=0\).
\item[\(\rm ii)\)] \(|v+w|\leq|v|+|w|\) for every \(v,w\in\mathscr M\).
\item[\(\rm iii)\)] \(|f\cdot v|=|f||v|\) for every \(f\in L^\infty(\mathcal N)\) and \(v\in\mathscr M\).
\item[\(\rm iv)\)] The \textbf{glueing property} is verified, i.e.\ whenever \(\{E_n\}_{n\in\N}\subseteq\Sigma\) is a partition
of \(\X\) and \(\{v_n\}_{n\in\N}\subseteq\mathscr M\) satisfies
\(\sup_{n\in\N}\|\pi_{\mathcal N}(\1_{E_n})|v_n|\|_{L^\infty(\mathcal N)}<+\infty\), there exists \(v\in\mathscr M\) such that
\begin{equation}\label{eq:glueing_property}
\pi_{\mathcal N}(\1_{E_n})\cdot v=\pi_{\mathcal N}(\1_{E_n})\cdot v_n\quad\text{ for every }n\in\N.
\end{equation}
The element \(v\), which is uniquely determined by \eqref{eq:glueing_property}, is denoted by
\(\sum_{n\in\N}\pi_{\mathcal N}(\1_{E_n})\cdot v_n\).
\end{itemize}
Whenever the norm \(\mathscr M\ni v\mapsto\||v|\|_{L^\infty(\mathcal N)}\) is complete,
we say that \(\mathscr M\) is an \textbf{\(L^\infty(\mathcal N)\)-Banach \(L^\infty(\mathcal N)\)-module}.
\end{definition}

The space \(L^\infty(\mathcal N)\) itself is an example of \(L^\infty(\mathcal N)\)-Banach \(L^\infty(\mathcal N)\)-module.
Other examples are the spaces of sections of a Banach bundle: if \({\bf E}\colon\X\twoheadrightarrow\mathbb U\)
is a separable Banach \(\mathbb U\)-bundle, then \(\bar\Gamma({\bf E})\) is an \(\mathcal L^\infty(\Sigma)\)-Banach \(\mathcal L^\infty(\Sigma)\)-module,
while \(\Gamma({\bf E})\) is an \(L^\infty(\mm)\)-Banach \(L^\infty(\mm)\)-module.
\begin{remark}\label{rmk:when_Linfty_is_R}{\rm
The theory of \(L^\infty(\mathcal N)\)-Banach \(L^\infty(\mathcal N)\)-modules is in fact an extension of the one of Banach spaces.
Indeed, if \((\X,\Sigma,\mm_A)\) is a probability space such that \(\X\) is an atom of \(\mm_A\), then \(L^\infty(\mm_A)\) can be
canonically identified with \(\R\), thus in particular the concept of \(L^\infty(\mm_A)\)-Banach \(L^\infty(\mm_A)\)-module
coincides with the one of Banach space.
We also point out that \((L^\infty(\mathcal N),L^\infty(\mathcal N),L^\infty(\mathcal N))\) is a metric
\(f\)-structure in the sense of \cite[Definition 2.23]{LP24} (cf.\ \cite[Section 4.2]{LP24}), thus we are in a position
to apply the results of \cite{LP24}.
\fr}\end{remark}

Two given \(L^\infty(\mathcal N)\)-Banach \(L^\infty(\mathcal N)\)-modules \(\mathscr M\) and \(\mathscr N\) are \textbf{isomorphic}
(as \(L^\infty(\mathcal N)\)-Banach \(L^\infty(\mathcal N)\)-modules), and we write \(\mathscr M\cong\mathscr N\), if there exists an
\(L^\infty(\mathcal N)\)-linear bijection \(\Phi\colon\mathscr M\to\mathscr N\) such that \(|\Phi(v)|=|v|\) for every \(v\in\mathscr M\).
Following \cite[Definition 3.5]{LP24}, we say that an \(L^\infty(\mathcal N)\)-Banach \(L^\infty(\mathcal N)\)-module \(\mathscr H\) is \textbf{Hilbertian}
provided it holds that
\begin{equation}\label{eq:ptwse_parall_law}
|v+w|^2+|v-w|^2=2|v|^2+2|w|^2\quad\text{ for every }v,w\in\mathscr H.
\end{equation}
We refer to \eqref{eq:ptwse_parall_law} as the \textbf{pointwise parallelogram law}. Note that, typically, a Hilbertian \(L^\infty(\mathcal N)\)-Banach
\(L^\infty(\mathcal N)\)-module is \emph{not} a Hilbert space; the Hilbertianity expressed by \eqref{eq:ptwse_parall_law} is a `fiberwise' condition
instead, cf.\ \cite[Theorem 3.1]{LPV22}. We define the \textbf{pointwise scalar product} in \(\mathscr H\) as
\[
v\cdot w\coloneqq\frac{|v+w|^2-|v|^2-|w|^2}{2}\in L^\infty(\mathcal N)\quad\text{ for every }v,w\in\mathscr H.
\]
It holds that \(\mathscr H\times\mathscr H\ni(v,w)\mapsto v\cdot w\in L^\infty(\mathcal N)\) is an \(L^\infty(\mathcal N)\)-bilinear map (and, in fact,
its \(L^\infty(\mathcal N)\)-bilinearity is equivalent to the Hilbertianity of \(\mathscr H\)).
\medskip

Let \((\X,\Sigma,\mm)\) be a \(\sigma\)-finite measure space and \(\mathscr M\) an \(\mathcal L^\infty(\Sigma)\)-Banach
\(\mathcal L^\infty(\Sigma)\)-module. Given any \(v,w\in\mathscr M\), we declare that \(v\sim w\) if and only if
\(|v-w|=0\) holds \(\mm\)-a.e.\ on \(\X\). Then the quotient
\begin{equation}\label{eq:quotient_mod}
\Pi_\mm(\mathscr M)\coloneqq\mathscr M/\sim
\end{equation}
inherits a natural structure of \(L^\infty(\mm)\)-Banach \(L^\infty(\mm)\)-module.
We denote by \(\pi_\mm\colon\mathscr M\to\Pi_\mm(\mathscr M)\) the canonical projection map.
\subsubsection*{Completion of a normed module}
Every \(L^\infty(\mathcal N)\)-normed \(L^\infty(\mathcal N)\)-module \(\mathscr M\) can be completed in a canonical way:
there exists a unique pair \((\bar{\mathscr M},\iota)\), where \(\bar{\mathscr M}\) is an \(L^\infty(\mathcal N)\)-Banach \(L^\infty(\mathcal N)\)-module
and \(\iota\colon\mathscr M\to\bar{\mathscr M}\) is an \(L^\infty(\mathcal N)\)-linear map with dense image such that \(|\iota(v)|=|v|\) for every \(v\in\mathscr M\).
Uniqueness is intended up a unique isomorphism, thanks to the validity of the following universal property: if \(\mathscr N\) is an \(L^\infty(\mathcal N)\)-Banach
\(L^\infty(\mathcal N)\)-module and \(T\colon\mathscr M\to\mathscr N\) is an \(L^\infty(\mathcal N)\)-linear map such that \(|T(v)|\leq|v|\) for every \(v\in\mathscr M\),
then there exists a unique \(L^\infty(\mathcal N)\)-linear map \(\bar T\colon\bar{\mathscr M}\to\mathscr N\) with \(|\bar T(v)|\leq|v|\) for every \(v\in\bar{\mathscr M}\) such that
\[\begin{tikzcd}
\mathscr M \arrow[r,"\iota"] \arrow[dr,swap,"T"] & \bar{\mathscr M} \arrow[d,"\bar T"] \\
& \mathscr N
\end{tikzcd}\]
is a commutative diagram. We say that \((\bar{\mathscr M},\iota)\) (or \(\bar{\mathscr M}\)) is the \textbf{module completion} of \(\mathscr M\).
We refer to \cite[Theorem 3.28]{LP24} for a proof of this statement.
\subsubsection*{Dual of a Banach module}
Let \((\X,\Sigma,\mm)\) be a \(\sigma\)-finite measure space and \(\mathscr M\) an \(L^\infty(\mm)\)-Banach \(L^\infty(\mm)\)-module.
Note that \(\mathbb D_{\mathscr M}=\{v\in\mathscr M:|v|\leq 1\}\). Then we denote by \(\mathscr M^*\) the space of all \(L^\infty(\mm)\)-linear
operators \(T\colon\mathscr M\to L^\infty(\mm)\) for which there exists a function \(g\in L^\infty(\mm)^+\) such that
\begin{equation}\label{eq:def_dual_mod}
|T(v)|\leq g|v|\quad\text{ for every }v\in\mathscr M.
\end{equation}
Given any element \(T\in\mathscr M^*\), we define its dual pointwise norm \(|T|\in L^\infty(\mm)^+\) as
\[
|T|\coloneqq\bigvee_{v\in\mathbb D_{\mathscr M}}|T(v)|=\bigwedge\big\{g\in L^\infty(\mm)^+\;\big|\;\eqref{eq:def_dual_mod}\text{ holds}\big\}.
\]
We have that the pair \((\mathscr M^*,|\cdot|)\) is an \(L^\infty(\mm)\)-Banach \(L^\infty(\mm)\)-module, which is called the \textbf{module dual}
of \(\mathscr M\). Typically, the module dual \(\mathscr M^*\) is not isomorphic (as a Banach space) to the dual Banach space \(\mathscr M'\).
For example, consider the Lebesgue space \(L^\infty(\mm)\) associated with an atomless probability space \((\X,\Sigma,\mm)\): it can be readily
checked that the module dual \(L^\infty(\mm)^*\) is isomorphic to \(L^\infty(\mm)\) itself, while the dual Banach space \(L^\infty(\mm)'\) is not.
We also point out that the above notion of module dual, which was introduced in \cite[Definition 4.16]{GLP22}, differs from the original one introduced
in \cite[Definition 1.2.6]{Gigli14}. Our choice is due to the fact that liftings of Banach modules can be defined only
for \(L^\infty(\mm)\)-Banach \(L^\infty(\mm)\)-modules (as we will see in Section \ref{sec:lifting}), thus it is convenient to consider a notion of
module dual that gives an \(L^\infty(\mm)\)-Banach \(L^\infty(\mm)\)-module.
\medskip

The pointwise duality pairing between \(\omega\in\mathscr M^*\) and \(v\in\mathscr M\) will be denoted by
\[
\langle\omega,v\rangle\coloneqq\omega(v)\in L^\infty(\mm).
\]
To any given element \(v\in\mathscr M\), we associate the set \({\rm Dual}(v)\subseteq\mathscr M^*\), which we define as
\[
{\rm Dual}(v)\coloneqq\left\{\omega\in\mathscr M^*\;\middle|\;|\omega|=\1_{\{|v|>0\}}^\mm,\;\langle\omega,v\rangle=|v|\right\}.
\]
A useful result that we will use many times is the following corollary of the Hahn--Banach extension theorem for \(L^\infty(\mm)\)-Banach \(L^\infty(\mm)\)-modules.
\begin{theorem}\label{thm:Hahn-Banach}
Let \((\X,\Sigma,\mm)\) be a \(\sigma\)-finite measure space and let \(\mathscr M\) be an \(L^\infty(\mm)\)-Banach
\(L^\infty(\mm)\)-module. Then it holds that \({\rm Dual}(v)\neq\varnothing\) for every \(v\in\mathscr M\).
\end{theorem}

Random generalisations of the Hahn--Banach separation theorem were obtained in \cite{Guo-1995}. In the context of \(L^p(\mm)\)-Banach \(L^\infty(\mm)\)-modules
for \(p<\infty\), the analogue of Theorem \ref{thm:Hahn-Banach} appeared implicitly in \cite{Gigli14}. For a proof of the statement in
Theorem \ref{thm:Hahn-Banach}, we refer to \cite[Theorem 3.30]{LP24}.
\subsubsection*{Representation of Banach modules as section spaces}
Let \((\X,\Sigma,\mathcal N)\) be an enhanced measurable space and let \(\mathscr M\) be an \(L^\infty(\mathcal N)\)-Banach \(L^\infty(\mathcal N)\)-module.
A set \(S\subseteq\mathscr M\) is said to \textbf{generate} \(\mathscr M\) (in the sense of \(L^\infty(\mathcal N)\)-Banach \(L^\infty(\mathcal N)\)-modules)
if the vector span of the elements of \(\mathscr M\) of the form \(\sum_{n\in\N}\pi_{\mathcal N}(\1_{E_n})\cdot v_n\), with
\(v_n\in S\), is dense in \(\mathscr M\). For example, \(L^\infty(\mathcal N)\) is generated by \(\{\pi_{\mathcal N}(\1_\X)\}\),
as it readily follows from the fact that measurable simple functions are dense in \(\mathcal L^\infty(\Sigma)\).
\medskip

When \(\mathscr M\) is generated by a countable set, we say that \(\mathscr M\) is \textbf{countably generated}. For example, if \(\mathbb U\) is a universal separable Banach
space and \({\bf E}\) is a separable Banach \(\mathbb U\)-bundle over \((\X,\Sigma)\), then \(\Gamma({\bf E})\) is countably generated for every \(\sigma\)-finite measure
\(\mm\) on \((\X,\Sigma)\). One of the main result of \cite{DMLP25} states that the converse holds as well: each countably-generated \(L^\infty(\mm)\)-Banach \(L^\infty(\mm)\)-module
is (isomorphic in the sense of \(L^\infty(\mm)\)-Banach \(L^\infty(\mm)\)-modules to) \(\Gamma({\bf E})\) for some separable Banach \(\mathbb U\)-bundle \({\bf E}\) over \((\X,\Sigma)\).
The precise statement, which is a consequence of \cite[Theorem 4.13]{DMLP25}, reads as follows.
\begin{theorem}[Representation theorem]
Let \((\X,\Sigma,\mm)\) be a \(\sigma\)-finite measure space and \(\mathbb U\) a universal separable Banach space. Let \(\mathscr M\) be a countably-generated \(L^\infty(\mm)\)-Banach
\(L^\infty(\mm)\)-module. Then there exists a separable Banach \(\mathbb U\)-bundle \({\bf E}\colon\X\twoheadrightarrow\mathbb U\) such that \(\Gamma({\bf E})\cong\mathscr M\).
\end{theorem}

Furthermore, the section functor \({\bf E}\mapsto\Gamma({\bf E})\) is an equivalence of categories between the category of separable Banach \(\mathbb U\)-bundles and the one of
countably-generated \(L^\infty(\mm)\)-Banach \(L^\infty(\mm)\)-modules. This result, which is reminiscent of the Serre--Swan theorem, was proved in \cite[Theorem 4.15]{DMLP25}.
Similar previous results for `locally finitely-generated modules' were obtained in \cite{LP18}.
\subsubsection*{Dual of a section space}
Let \((\X,\Sigma,\mm)\) be a \(\sigma\)-finite measure space. Let \(\mathbb U\) be a universal separable Banach space and \({\bf E}\) a separable Banach \(\mathbb U\)-bundle
over \((\X,\Sigma)\). Then it is possible to characterise the module dual \(\Gamma({\bf E})^*\) of \(\Gamma({\bf E})\) in a fiberwise manner, as we are going to describe.
First, we denote by \(\bar\Gamma({\bf E}'_{w^*})\) the set of all \(\omega(\star)\in\prod_{x\in\X}{\bf E}(x)'\) such that \(\sup_{x\in\X}\|\omega(x)\|_{{\bf E}(x)'}<+\infty\) and
\begin{equation}\label{eq:def_E_wstar}
\X\ni x\mapsto\langle\omega(x),v(x)\rangle\in\R\quad\text{ is measurable for every }v\in\bar\Gamma({\bf E}).
\end{equation}
It follows from \eqref{eq:def_E_wstar} that \(\|\omega(\star)\|_{{\bf E}(\star)'}\in\mathcal L^\infty(\Sigma)\)
for every \(\omega\in\bar\Gamma({\bf E}'_{w^*})\). Next, we introduce an equivalence relation \(\sim\) on
\(\bar\Gamma({\bf E}'_{w^*})\) by declaring, for any given \(\omega,\eta\in\bar\Gamma({\bf E}'_{w^*})\), that
\[
\omega\sim\eta\quad\Longleftrightarrow\quad\omega(x)=\eta(x)\text{ for }\mm\text{-a.e.\ }x\in\X.
\]
Finally, we define \(\Gamma({\bf E}'_{w^*})\) as the quotient space
\[
\Gamma({\bf E}'_{w^*})\coloneqq\bar\Gamma({\bf E}'_{w^*})/\sim.
\]
It holds that \(\Gamma({\bf E}'_{w^*})\) is an \(L^\infty(\mm)\)-Banach \(L^\infty(\mm)\)-module with respect to the natural pointwise operations
and the following pointwise norm:
\[
|\omega|\coloneqq\pi_\mm\big(\|\bar\omega(\star)\|_{{\bf E}(\star)'}\big)\quad\text{ for every }\omega\in\Gamma({\bf E}'_{w^*}),
\]
where \(\bar\omega\in\bar\Gamma({\bf E}'_{w^*})\) is any representative of \(\omega\); this definition is well posed,
as it is independent of the choice of \(\bar\omega\). The space \(\Gamma({\bf E}'_{w^*})\) constitutes an explicit
description of the module dual of \(\Gamma({\bf E})\):
\begin{theorem}[\(\Gamma({\bf E})^*\cong\Gamma({\bf E}'_{w^*})\)]\label{thm:dual_section_space}
Let \((\X,\Sigma,\mm)\) be a \(\sigma\)-finite measure space. Let \(\mathbb U\) be a universal separable Banach space
and \({\bf E}\) a separable Banach \(\mathbb U\)-bundle over \((\X,\Sigma)\). Then \(\Gamma({\bf E})^*\) and
\(\Gamma({\bf E}'_{w^*})\) are isomorphic as \(L^\infty(\mm)\)-Banach \(L^\infty(\mm)\)-modules, through the
isomorphism \(\Phi\colon\Gamma({\bf E}'_{w^*})\to\Gamma({\bf E})^*\) that is defined as follows:
\[
\langle\Phi(\omega),v\rangle\coloneqq\pi_\mm\big(\langle\bar\omega(\star),\bar v(\star)\rangle\big)
\quad\text{ for every }\omega\in\Gamma({\bf E}'_{w*})\text{ and }v\in\Gamma({\bf E}),
\]
where \(\bar\omega\in\bar\Gamma({\bf E}'_{w^*})\) and \(\bar v\in\bar\Gamma({\bf E})\) are arbitrarily
chosen representatives of \(\omega\) and \(v\), respectively.
\end{theorem}

The above statements were originally obtained (for the space of \(L^0(\mm)\)-sections of \(\bf E\)) in \cite{LPV22}
(see also \cite{GLP22} for related results). More precisely, Theorem \ref{thm:dual_section_space} can be proved by
suitably adapting the proof of \cite[Theorem 3.8]{LPV22}, by taking also \cite[Remarks 3.6 and 3.7]{LPV22} into account.
\subsection{Bounded \texorpdfstring{\(L^\infty(\mathcal N)\)}{Linfty(N)}-multilinear operators}
Let us now introduce the space of bounded \(L^\infty(\mathcal N)\)-multilinear operators between
\(L^\infty(\mathcal N)\)-normed \(L^\infty(\mathcal N)\)-modules.
\begin{definition}[Bounded \(L^\infty(\mathcal N)\)-multilinear operators]\label{def:bounded_Linf_multilinear}
Let \((\X,\Sigma,\mathcal N)\) be an enhanced measurable space and \(k\in\N\). Let \(\mathscr M_1,\ldots,\mathscr M_k,\mathscr N\) be \(L^\infty(\mathcal N)\)-normed
\(L^\infty(\mathcal N)\)-modules. Then we denote by \({\rm M}_k(\mathscr M_1,\ldots,\mathscr M_k;\mathscr N)\) the space
of all \(L^\infty(\mathcal N)\)-multilinear operators \(\Phi\colon\mathscr M_1\times\ldots\times\mathscr M_k\to\mathscr N\)
for which there exists a function \(g\in L^\infty(\mathcal N)^+\) such that
\begin{equation}\label{eq:def_bdd_Linfty_lin}
|\Phi(v_1,\ldots,v_k)|\leq g|v_1|\ldots|v_k|\quad\text{ for every }(v_1,\ldots,v_k)\in\mathscr M_1\times\ldots\times\mathscr M_k.
\end{equation}
The elements of \({\rm M}_k(\mathscr M_1,\ldots,\mathscr M_k;\mathscr N)\) are said to be
\textbf{bounded \(L^\infty(\mathcal N)\)-multilinear operators}.
\end{definition}
The notation \({\rm M}_k(\mathscr M_1,\ldots,\mathscr M_k;\mathscr N)\) is ambiguous, as it does not specify whether
\(\mathscr M_1,\ldots,\mathscr M_k,\mathscr N\) are considered as \(L^\infty(\mathcal N)\)-Banach \(L^\infty(\mathcal N)\)-modules
or as Banach spaces. Nevertheless, in order to keep the notation as simple as possible, we still write \({\rm M}_k(\mathscr M_1,\ldots,\mathscr M_k;\mathscr N)\)
to mean the space we introduced in Definition \ref{def:bounded_Linf_multilinear} whenever it is clear from the context that
\(\mathscr M_1,\ldots,\mathscr M_k,\mathscr N\) are regarded as \(L^\infty(\mathcal N)\)-Banach \(L^\infty(\mathcal N)\)-modules.
\medskip

In fact, we will be concerned with only two classes of bounded \(L^\infty(\mathcal N)\)-multilinear operators:
\begin{itemize}
\item For modules over \(L^\infty(\mm)\), where \((\X,\Sigma,\mm)\) is a \(\sigma\)-finite measure space.
In this case, we set
\[
|\Phi|\coloneqq\bigvee_{(v_1,\ldots,v_k)\in\mathbb D_{\mathscr M_1}\times\ldots\times\mathbb D_{\mathscr M_k}}|\Phi(v_1,\ldots,v_k)|\in L^\infty(\mm)^+
\]
for every \(\Phi\in{\rm M}_k(\mathscr M_1,\ldots,\mathscr M_k;\mathscr N)\). Note that \(|\Phi|=\bigwedge\{g\in L^\infty(\mm)^+:\eqref{eq:def_bdd_Linfty_lin}\text{ holds}\}\).
We have that \(\big({\rm M}_k(\mathscr M_1,\ldots,\mathscr M_k;\mathscr N),|\cdot|\big)\) is an \(L^\infty(\mm)\)-normed
\(L^\infty(\mm)\)-module. Moreover, if \(\mathscr N\) is an \(L^\infty(\mm)\)-Banach
\(L^\infty(\mm)\)-module, then \({\rm M}_k(\mathscr M_1,\ldots,\mathscr M_k;\mathscr N)\) is an \(L^\infty(\mm)\)-Banach
\(L^\infty(\mm)\)-module as well.
\item For modules over \(\mathcal L^\infty(\Sigma)\). In this case, for any given
\(\Phi\in{\rm M}_k(\mathscr M_1,\ldots,\mathscr M_k;\mathscr N)\) we set
\[
|\Phi|(x)\coloneqq\sup_{(v_1,\ldots,v_k)\in\mathbb D_{\mathscr M_1}\times\ldots\times\mathbb D_{\mathscr M_k}}|\Phi(v_1,\ldots,v_k)|(x)
\quad\text{ for every }x\in\X.
\]
Note that \(|\Phi|\) is not necessarily an element of \(\mathcal L^\infty(\Sigma)\), since it might be non-measurable.
\end{itemize}
We will focus mostly on the cases \(k=1\) and \(k=2\), where we use the following notations:
\[
\textsc{Hom}(\mathscr M_1;\mathscr N)\coloneqq{\rm M}_1(\mathscr M_1;\mathscr N),\qquad
{\rm B}(\mathscr M_1,\mathscr M_2;\mathscr N)\coloneqq{\rm M}_2(\mathscr M_1,\mathscr M_2;\mathscr N).
\]
Moreover, we use the shorthand notation \({\rm B}(\mathscr M_1,\mathscr M_2)\coloneqq{\rm B}(\mathscr M_1,\mathscr M_2;L^\infty(\mathcal N))\).
Note also that, when \(\mathcal N=\mathcal N_\mm\) for some \(\sigma\)-finite measure space \((\X,\Sigma,\mm)\),
we have \(\mathscr M_1^*=\textsc{Hom}(\mathscr M_1;L^\infty(\mm))\).
\medskip

The following result can be obtained by suitably adapting e.g.\ the proof of \cite[Theorem 3.16]{LP24}.
\begin{proposition}\label{prop:extension_bdd_multilin}
Let \((\X,\Sigma,\mathcal N)\) be an enhanced measurable space. Let \(\mathscr M_1,\ldots,\mathscr M_k,\mathscr N\) be
\(L^\infty(\mathcal N)\)-Banach \(L^\infty(\mathcal N)\)-modules. Let \(V_i\) be a generating \(\mathbb Q\)-vector subspace of \(\mathscr M_i\)
for any \(i=1,\ldots,k\). Assume \(\varphi\colon V_1\times\ldots\times V_k\to\mathscr N\) is a \(\mathbb Q\)-multilinear map for which
there exists a function \(g\in L^\infty(\mathcal N)^+\) such that
\[
|\varphi(v_1,\ldots,v_k)|\leq g|v_1|\ldots|v_k|\quad\text{ for every }(v_1,\ldots,v_k)\in V_1\times\ldots\times V_k.
\]
Then there exists a unique extension \(\Phi\in{\rm M}_k(\mathscr M_1,\ldots,\mathscr M_k;\mathscr N)\) of \(\varphi\). Moreover, it holds that
\[
|\Phi(v_1,\ldots,v_k)|\leq g|v_1|\ldots|v_k|\quad\text{ for every }(v_1,\ldots,v_k)\in\mathscr M_1\times\ldots\times\mathscr M_k.
\]
\end{proposition}
\subsection{Tensor products of \texorpdfstring{\(L^\infty(\mm)\)}{Linfty(m)}-Banach \texorpdfstring{\(L^\infty(\mm)\)}{Linfty(m)}-modules}
In the paper \cite{Pas23}, the third named author introduced and studied tensor products of
\(L^0(\mm)\)-Banach \(L^0(\mm)\)-modules. In this section, we will formulate several concepts and results
concerning tensor products in the framework of \(L^\infty(\mm)\)-Banach \(L^\infty(\mm)\)-modules instead.
We shall not spend many words neither on the well posedness of the definitions nor on the proofs of the
stated results, since they can be shown by repeating almost verbatim the arguments presented in \cite{Pas23}.
\begin{theorem}[Injective pointwise norm]
Let \((\X,\Sigma,\mm)\) be a \(\sigma\)-finite measure space. Let \(\mathscr M\) and \(\mathscr N\) be \(L^\infty(\mm)\)-Banach
\(L^\infty(\mm)\)-modules. Then \(\mathscr M\otimes_\varepsilon\mathscr N=(\mathscr M\otimes\mathscr N,|\cdot|_\varepsilon)\)
is an \(L^\infty(\mm)\)-normed \(L^\infty(\mm)\)-module, where the \textbf{injective pointwise norm} \(|\cdot|_\varepsilon\)
is defined as
\[
|\alpha|_\varepsilon\coloneqq\bigvee\bigg\{\sum_{i=1}^n\langle\omega,v_i\rangle\langle\eta,w_i\rangle\;\bigg|
\;\omega\in\mathbb D_{\mathscr M^*},\,\eta\in\mathbb D_{\mathscr N^*}\bigg\}
\]
for every \(\alpha=\sum_{i=1}^n v_i\otimes w_i\in\mathscr M\otimes\mathscr N\).
\end{theorem}

It holds that \(|v\otimes w|_\varepsilon=|v||w|\) for every \(v\in\mathscr M\) and \(w\in\mathscr N\).
\begin{definition}[Injective tensor products of \(L^\infty(\mm)\)-Banach \(L^\infty(\mm)\)-modules]
Let \((\X,\Sigma,\mm)\) be a \(\sigma\)-finite measure space. Then we define the \(L^\infty(\mm)\)-Banach
\(L^\infty(\mm)\)-module \(\mathscr M\hat\otimes_\varepsilon\mathscr N\) as the module completion of
\(\mathscr M\otimes_\varepsilon\mathscr N\). We say that \(\mathscr M\hat\otimes_\varepsilon\mathscr N\)
is the \textbf{injective tensor product} of \(\mathscr M\) and \(\mathscr N\).
\end{definition}

Given that \(L^\infty(\mm)\)-Banach \(L^\infty(\mm)\)-modules are in particular Banach spaces, one should specify whether the
injective tensor product under consideration is the one \(\mathscr M\hat\otimes_\varepsilon^{L^\infty(\mm)}\mathscr N\)
in the sense of \(L^\infty(\mm)\)-Banach \(L^\infty(\mm)\)-modules or the one \(\mathscr M\hat\otimes_\varepsilon^\R\mathscr N\) in the sense
of Banach spaces. In fact, as Example \ref{ex:diff_inj_prod} will show, it can happen that \(\mathscr M\hat\otimes_\varepsilon^{L^\infty(\mm)}\mathscr N\)
(as a Banach space) differs from \(\mathscr M\hat\otimes_\varepsilon^\R\mathscr N\). Nevertheless, to avoid a cumbersome notation,
we just write \(\mathscr M\hat\otimes_\varepsilon\mathscr N\coloneqq\mathscr M\hat\otimes_\varepsilon^{L^\infty(\mm)}\mathscr N\)
whenever it is tacitly understood that \(\mathscr M\) and \(\mathscr N\) are regarded as \(L^\infty(\mm)\)-Banach \(L^\infty(\mm)\)-modules.
\begin{example}\label{ex:diff_inj_prod}{\rm
It can be readily verified that \(L^\infty(\mm)\hat\otimes_\varepsilon L^\infty(\mm)=L^\infty(\mm)\otimes_\varepsilon L^\infty(\mm)\cong L^\infty(\mm)\)
in the sense of \(L^\infty(\mm)\)-Banach \(L^\infty(\mm)\)-modules, via the isomorphism
\[
L^\infty(\mm)\otimes_\varepsilon L^\infty(\mm)\ni\sum_{i=1}^n f_i\otimes g_i\mapsto\sum_{i=1}^n f_i g_i\in L^\infty(\mm).
\]
We now consider the particular case where \((\X,\Sigma)\coloneqq(\N,2^\N)\) and \(\mm\coloneqq\sum_{n\in\N}\delta_n\),
thus \(L^\infty(\mm)=\ell^\infty\). It is known that \(\ell^\infty\hat\otimes_\varepsilon^\R\ell^\infty\) and \(\ell^\infty\)
are not linearly isomorphic. We observed that \(\ell^\infty\hat\otimes_\varepsilon^{\ell^\infty}\ell^\infty\cong\ell^\infty\),
thus we can conclude that \(\ell^\infty\hat\otimes_\varepsilon^\R\ell^\infty\) and \(\ell^\infty\hat\otimes_\varepsilon^{\ell^\infty}\ell^\infty\)
are not linearly isomorphic (so that, a fortiori, they are not isomorphic as Banach spaces).
\fr}\end{example}
\begin{remark}\label{rmk:equiv_inj_tens_norm}{\rm
Given any \(\alpha=\sum_{i=1}^n v_i\otimes w_i\in\mathscr M\otimes_\varepsilon\mathscr N\) and \(\delta>0\), there exist \(\omega\in\mathbb D_{\mathscr M^*}\) and \(\eta\in\mathbb D_{\mathscr N^*}\) such that
\begin{equation}\label{eq:equiv_inj_tens_norm}
\sum_{i=1}^n\langle\omega,v_i\rangle\langle\eta,w_i\rangle\geq|\alpha|_\varepsilon-\delta.
\end{equation}
Indeed, we can find a partition \((E_k)_{k\in\N}\subseteq\Sigma\) of \(\X\) and a sequence \(((\omega_k,\eta_k))_{k\in\N}\subseteq\mathbb D_{\mathscr M^*}\times\mathbb D_{\mathscr N^*}\)
such that \(\1_{E_k}^\mm\sum_{i=1}^n\langle\omega_k,v_i\rangle\langle\eta_k,w_i\rangle\geq\1_{E_k}^\mm(|\alpha|_\varepsilon-\delta)\) holds for every \(k\in\N\). It then follows that the elements
\(\omega\coloneqq\sum_{k\in\N}\1_{E_k}^\mm\cdot\omega_k\in\mathbb D_{\mathscr M^*}\) and \(\eta\coloneqq\sum_{k\in\N}\1_{E_k}^\mm\cdot\eta_k\in\mathbb D_{\mathscr N^*}\)
verify \eqref{eq:equiv_inj_tens_norm}.
\fr}\end{remark}
\begin{theorem}[Projective pointwise norm]
Let \((\X,\Sigma,\mm)\) be a \(\sigma\)-finite measure space. Let \(\mathscr M\) and \(\mathscr N\) be \(L^\infty(\mm)\)-Banach
\(L^\infty(\mm)\)-modules. Then \(\mathscr M\otimes_\pi\mathscr N=(\mathscr M\otimes\mathscr N,|\cdot|_\pi)\)
is an \(L^\infty(\mm)\)-normed \(L^\infty(\mm)\)-module, where the \textbf{projective pointwise norm} \(|\cdot|_\pi\)
is defined as
\[
|\alpha|_\pi\coloneqq\bigwedge\left\{\sum_{i=1}^n|v_i||w_i|\;\middle|\;(v_i)_{i=1}^n\subseteq\mathscr M,\,
(w_i)_{i=1}^n\subseteq\mathscr N,\,\alpha=\sum_{i=1}^n v_i\otimes w_i\right\}
\]
for every \(\alpha\in\mathscr M\otimes\mathscr N\).
\end{theorem}

It holds that \(|v\otimes w|_\pi=|v||w|\) for every \(v\in\mathscr M\) and \(w\in\mathscr N\).
\begin{definition}[Projective tensor products of \(L^\infty(\mm)\)-Banach \(L^\infty(\mm)\)-modules]
Let \((\X,\Sigma,\mm)\) be a \(\sigma\)-finite measure space. Then we define the \(L^\infty(\mm)\)-Banach
\(L^\infty(\mm)\)-module \(\mathscr M\hat\otimes_\pi\mathscr N\) as the module completion of
\(\mathscr M\otimes_\pi\mathscr N\). We say that \(\mathscr M\hat\otimes_\pi\mathscr N\)
is the \textbf{projective tensor product} of \(\mathscr M\) and \(\mathscr N\).
\end{definition}

Similarly to the above discussion for injective tensor products, we prefer to use the shorthand notation
\(\mathscr M\hat\otimes_\pi\mathscr N\coloneqq\mathscr M\hat\otimes_\pi^{L^\infty(\mm)}\mathscr N\) for the projective tensor product
in the sense of \(L^\infty(\mm)\)-Banach \(L^\infty(\mm)\)-modules. The following example shows that
\(\mathscr M\hat\otimes_\pi^{L^\infty(\mm)}\mathscr N\) and \(\mathscr M\hat\otimes_\pi^\R\mathscr N\) can be different.
\begin{example}\label{ex:diff_proj_prod}{\rm
It can be readily verified that \(L^\infty(\mm)\hat\otimes_\pi L^\infty(\mm)=L^\infty(\mm)\otimes_\pi L^\infty(\mm)\cong L^\infty(\mm)\)
in the sense of \(L^\infty(\mm)\)-Banach \(L^\infty(\mm)\)-modules, via the isomorphism
\[
L^\infty(\mm)\otimes_\pi L^\infty(\mm)\ni\sum_{i=1}^n f_i\otimes g_i\mapsto\sum_{i=1}^n f_i g_i\in L^\infty(\mm).
\]
Moreover, it is known that \(\ell^\infty\hat\otimes_\pi^\R\ell^\infty\) contains a complemented copy
of \(\ell^2\). On the other hand, \(\ell^\infty\) is a \textbf{prime}
Banach space (i.e.\ all its complemented infinite-dimensional subspaces are isomorphic to \(\ell^\infty\), cf.\ \cite[Definition 2.2.5]{AlbiacKalton}),
as it was proved in \cite{Lindenstrauss1967} (see also \cite[Theorem 5.6.5]{AlbiacKalton}). Hence, since the Banach spaces \(\ell^2\) and \(\ell^\infty\)
are not isomorphic (as the former is separable, while the latter is not), we conclude that \(\ell^\infty\hat\otimes_\pi^\R\ell^\infty\)
and \(\ell^\infty\hat\otimes_\pi^{\ell^\infty}\ell^\infty\cong\ell^\infty\) are not isomorphic as Banach spaces.
\fr}\end{example}
\begin{theorem}[Universal property of \(\mathscr M\hat\otimes_\pi\mathscr N\)]\label{thm:proj_tens_Ban_mod}
Let \((\X,\Sigma,\mm)\) be a \(\sigma\)-finite measure space. Let \(\mathscr M\), \(\mathscr N\), \(\mathscr Q\)
be \(L^\infty(\mm)\)-Banach \(L^\infty(\mm)\)-modules. Let \(b\in{\rm B}(\mathscr M,\mathscr N;\mathscr Q)\) be
given. Then there exists a unique operator \(\tilde b_\pi\in\textsc{Hom}(\mathscr M\hat\otimes_\pi\mathscr N;\mathscr Q)\)
such that
\[\begin{tikzcd}
\mathscr M\times\mathscr N \arrow[d,swap,"\otimes"] \arrow[r,"b"] & \mathscr Q \\
\mathscr M\hat\otimes_\pi\mathscr N \arrow[ur,swap,"\tilde b_\pi"] &
\end{tikzcd}\]
is a commutative diagram. Moreover, the resulting map
\[
{\rm B}(\mathscr M,\mathscr N;\mathscr Q)\ni b\mapsto\tilde b_\pi\in\textsc{Hom}(\mathscr M\hat\otimes_\pi\mathscr N;\mathscr Q)
\]
is an isomorphism of \(L^\infty(\mm)\)-Banach \(L^\infty(\mm)\)-modules.
\end{theorem}

It follows from Theorem \ref{thm:proj_tens_Ban_mod} that the couple \((\mathscr M\hat\otimes_\pi\mathscr N,\otimes)\) is unique
up to a unique isomorphism, meaning that for any pair \((\mathscr Q,\mathfrak p)\) having the same properties there exists a unique
isomorphism of \(L^\infty(\mm)\)-Banach \(L^\infty(\mm)\)-modules \(\Phi\colon\mathscr M\hat\otimes_\pi\mathscr N\to\mathscr Q\) such that
\[\begin{tikzcd}
\mathscr M\times\mathscr N \arrow[r,"\otimes"] \arrow[dr,swap,"\mathfrak p"] & \mathscr M\hat\otimes_\pi\mathscr N \arrow[d,"\Phi"] \\
& \mathscr Q 
\end{tikzcd}\]
is a commutative diagram. In this case, we write \((\mathscr M\hat\otimes_\pi\mathscr N,\otimes)\cong(\mathscr Q,\mathfrak p)\).
\begin{theorem}[The dual of \(\mathscr M\hat\otimes_\pi\mathscr N\)]\label{thm:dual_proj_tens_prod}
Let \((\X,\Sigma,\mm)\) be a \(\sigma\)-finite measure space. Let \(\mathscr M\) and \(\mathscr N\)
be \(L^\infty(\mm)\)-Banach \(L^\infty(\mm)\)-modules. Then it holds that
\[
{\rm B}(\mathscr M,\mathscr N)\cong(\mathscr M\hat\otimes_\pi\mathscr N)^*
\]
via the isomorphism \({\rm B}(\mathscr M,\mathscr N)\ni b\mapsto\tilde b_\pi\in(\mathscr M\hat\otimes_\pi\mathscr N)^*\).
In particular, it holds that
\[
|\alpha|_\pi=\bigvee\left\{\tilde b_\pi(\alpha)\;\middle|\;b\in\mathbb D_{{\rm B}(\mathscr M,\mathscr N)}\right\}
\quad\text{ for every }\alpha\in\mathscr M\hat\otimes_\pi\mathscr N.
\]
\end{theorem}

It is worth pointing out that the results of this section imply the corresponding ones for Banach spaces.
Indeed, by taking Remark \ref{rmk:when_Linfty_is_R} into account, it is easy to check that the above definitions and
results reduce to the classical ones for Banach spaces when the reference probability measure under consideration
consists of a single atom. In this sense, the theory of tensor products of \(L^\infty(\mm)\)-Banach \(L^\infty(\mm)\)-modules
is an actual extension of that of Banach spaces.
\subsection{Liftings of \texorpdfstring{\(L^\infty(\mm)\)}{Linfty(m)}-Banach \texorpdfstring{\(L^\infty(\mm)\)}{Linfty(m)}-modules}\label{sec:lifting}
Let us now discuss the theory of liftings of \(L^\infty(\mm)\)-Banach \(L^\infty(\mm)\)-modules,
which was first introduced in \cite{DMLP25}. As we have mentioned in the Introduction, our choice of working exactly 
with \(L^\infty(\mm)\)-Banach \(L^\infty(\mm)\)-modules (instead e.g.\ of  \(L^0(\mm)\)-Banach \(L^0(\mm)\)-modules)
is due to the fact that they are well suited to the lifting theory.
\medskip

Fix a complete probability space \((\X,\Sigma,\mm)\) and a lifting \(\ell\) of \(\mm\). Let \(\mathscr M\) be an \(L^\infty(\mm)\)-Banach \(L^\infty(\mm)\)-module.
Then there exists a unique couple \((\ell\mathscr M,\ell)\) such that the following conditions hold:
\begin{itemize}
\item \(\ell\mathscr M\) is an \(\mathcal L^\infty(\Sigma)\)-Banach \(\mathcal L^\infty(\Sigma)\)-module,
\item \(\ell\colon\mathscr M\to\ell\mathscr M\) is a linear operator such that \(|\ell v|=\ell|v|\) for every \(v\in\mathscr M\),
\item the image of \(\ell\colon\mathscr M\to\ell\mathscr M\) generates \(\ell\mathscr M\).
\end{itemize}
The uniqueness of \((\ell\mathscr M,\ell)\) is up to a unique isomorphism. More generally, the following universal
property holds: given any \(\mathcal L^\infty(\Sigma)\)-Banach \(\mathcal L^\infty(\Sigma)\)-module \(\mathscr N\)
and any linear operator \(t\colon\mathscr M\to\mathscr N\) such that \(|tv|\leq\ell|v|\) for every \(v\in\mathscr M\),
there exists a unique \(T\in\textsc{Hom}(\ell\mathscr M;\mathscr N)\) such that
\[\begin{tikzcd}
\mathscr M \arrow[r,"\ell"] \arrow[dr,swap,"t"] & \ell\mathscr M \arrow[d,"T"] \\
& \mathscr N
\end{tikzcd}\]
is a commutative diagram. We say that \((\ell\mathscr M,\ell)\) is the \textbf{lifting} of \(\mathscr M\). Liftings of \(L^\infty(\mm)\)-Banach
\(L^\infty(\mm)\)-modules were introduced in \cite[Theorem 3.5]{DMLP25}, where existence and uniqueness were proved, while the universal property
follows from \cite[Corollary 3.21 and Section 4.2.5]{LP24}. It holds that
\begin{equation}\label{eq:ell_inv_a.e.}
\mathscr M\cong\Pi_\mm(\ell\mathscr M),
\end{equation}
where the \(L^\infty(\mm)\)-Banach \(L^\infty(\mm)\)-module \(\Pi_\mm(\ell\mathscr M)\) is defined
as in \eqref{eq:quotient_mod}.
\begin{remark}\label{rmk:lift_Linfty(mm)}{\rm
We claim that
\[
\ell L^\infty(\mm)\cong\mathcal L^\infty(\Sigma).
\]
More precisely, we have that the pair \((\mathcal L^\infty(\Sigma),\ell)\) is the lifting of \(L^\infty(\mm)\), differently from what it was
erroneously stated in \cite[Example 4.8]{GLP22}. Indeed, \(\mathcal L^\infty(\Sigma)\) is an \(\mathcal L^\infty(\Sigma)\)-Banach
\(\mathcal L^\infty(\Sigma)\)-module, the lifting \(\ell\colon L^\infty(\mm)\to\mathcal L^\infty(\Sigma)\) is a linear map satisfying
\(|\ell f|=\ell|f|\) for every \(f\in L^\infty(\mm)\), and the image of \(\ell\colon L^\infty(\mm)\to\mathcal L^\infty(\Sigma)\)
contains the function \(\1_\X=\ell\1_\X^\mm\), which generates \(\mathcal L^\infty(\Sigma)\).
\fr}\end{remark}
\subsubsection*{Fibers of a lifted module}
Let \((\X,\Sigma,\mm)\) be a complete probability space and \(\ell\) a lifting of \(\mm\). Fix an \(L^\infty(\mm)\)-Banach
\(L^\infty(\mm)\)-module \(\mathscr M\) and a point \(x\in\X\). Then we define \(p_x\colon\ell\mathscr M\to[0,+\infty)\)
as follows: given any element \(v\in\ell\mathscr M\), we set
\begin{equation}\label{eq:def_p_x}
p_x(v)\coloneqq\left\{\begin{array}{ll}
|v|(x)\\
\ell(\pi_\mm(|v|))(x)
\end{array}\quad\begin{array}{ll}
\text{ if }x\in\X\setminus{\rm A}(\ell),\\
\text{ if }x\in{\rm A}(\ell).
\end{array}\right.
\end{equation}
It can be readily checked that \(p_x\) is a complete seminorm on the vector space \(\ell\mathscr M\).
\begin{definition}[Fibers of a lifted module]
We define the \textbf{fiber} of \(\ell\mathscr M\) at \(x\) as the Banach space
\[
\ell\mathscr M_x\coloneqq\ell\mathscr M/p_x.
\]
Given any \(v\in\mathscr M\), we define its \textbf{value}
\(\ell v_x\in\ell\mathscr M_x\) at \(x\) as the equivalence class of \(\ell v\) in \(\ell\mathscr M_x\).
\end{definition}

Observe that
\[
\|\ell v_x\|_{\ell\mathscr M_x}=\ell|v|(x)\quad\text{ for every }v\in\mathscr M\text{ and }x\in\X.
\]
Note also that \(\mathscr M\ni v\mapsto\ell v_x\in\ell\mathscr M_x\) is linear, thus
\(\{\ell v_x:v\in\mathscr M\}\) is a vector subspace of \(\ell\mathscr M_x\).
\begin{lemma}\label{lem:dens_in_fiber_of_lift_mod}
Fix any \(x\in\X\). Then \(\{\ell v_x:v\in\mathscr M\}\) is a dense vector subspace of \(\ell\mathscr M_x\).
\end{lemma}
\begin{proof}
Let \(w\in\ell\mathscr M_x\) and \(\varepsilon>0\) be given. Pick some representative \(\bar w\in\ell\mathscr M\) of \(w\).
We can find a partition \((E_n)_{n\in\N}\subseteq\Sigma\) of \(\X\) and a sequence \((v^n)_{n\in\N}\subseteq\mathscr M\)
such that \(\big\|\bar w-\sum_{n\in\N}\1_{E_n}\cdot\ell v^n\big\|_{\ell\mathscr M}\leq\varepsilon\).
We distinguish two cases. First, assume \(x\in\X\setminus{\rm A}(\ell)\). There exists (a unique) index \(n_x\in\N\)
such that \(x\in E_{n_x}\). Hence, we have that \(p_x\big(\ell v^{n_x}-\sum_{n\in\N}\1_{E_n}\cdot\ell v^n\big)=0\)
and thus accordingly
\[
\|w-\ell v^{n_x}_x\|_{\ell\mathscr M_x}=p_x(\bar w-\ell v^{n_x})=p_x\left(\bar w-\sum_{n\in\N}\1_{E_n}\cdot\ell v^n\right)
\leq\left\|\bar w-\sum_{n\in\N}\1_{E_n}\cdot\ell v^n\right\|_{\ell\mathscr M}\leq\varepsilon,
\]
which yields the statement in the case \(x\notin{\rm A}(\ell)\). Finally, assume \(x\in A\) for some \(A\in\mathscr A(\ell)\).
By Lemma \ref{lem:lifted_atoms} ii), there exists (a unique) \(n_x\in\N\) such that \(A\subseteq\ell E_{n_x}\).
By Lemma \ref{lem:lifted_atoms} iii), we get
\[\begin{split}
\|w-\ell v^{n_x}_x\|_{\ell\mathscr M_x}&=p_x(\bar w-\ell v^{n_x})=\ell(\pi_\mm(|\bar w-\ell v^{n_x}|))(x)
=\big\|\1_{A\cap E_{n_x}}^\mm\pi_\mm(|\bar w-\ell v^{n_x}|)\big\|_{L^\infty(\mm)}\\
&\leq\left\|\bar w-\sum_{n\in\N}\1_{E_n}\cdot\ell v^n\right\|_{\ell\mathscr M}\leq\varepsilon,
\end{split}\]
which yields the statement in the case \(x\in{\rm A}(\ell)\). Therefore, the proof is complete.
\end{proof}
\begin{remark}{\rm
Let \(A\in\mathscr A(\ell)\) be given. Then we claim that
\begin{subequations}\begin{align}
\label{eq:const_vect_on_atoms1}
\ell\mathscr M_x=\ell\mathscr M_y&\quad\text{ for every }x,y\in A,\\
\label{eq:const_vect_on_atoms2}
\ell v_x=\ell v_y&\quad\text{ for every }v\in\mathscr M\text{ and }x,y\in A.
\end{align}\end{subequations}
Indeed, for any \(v\in\mathscr M\) we deduce from Lemma \ref{lem:lifted_atoms} iii) that \(\ell(\pi_\mm(|v|))\) is constant on \(A\).
This means that \(p_x(v)=p_y(v)\) for all \(v\in\mathscr M\) and \(x,y\in A\), whence \eqref{eq:const_vect_on_atoms1} and \eqref{eq:const_vect_on_atoms2} follow.
\fr}\end{remark}
\begin{remark}\label{rmk:Hilbertian_lifting}{\rm
Assume that \(\mathscr H\) is a Hilbertian \(L^\infty(\mm)\)-Banach \(L^\infty(\mm)\)-module. Then it is easy to see that its lifting
\(\ell\mathscr H\) is a Hilbertian \(\mathcal L^\infty(\Sigma)\)-Banach \(\mathcal L^\infty(\Sigma)\)-module. In particular, we deduce
that the fiber \(\ell\mathscr H_x\) is a Hilbert space for every \(x\in\X\).
\fr}\end{remark}

The following result, which was proved in \cite[Proposition 3.10]{DMLP25}, states that the fibers of \(\ell(\mathscr M^*)\)
are subspaces of the duals of the corresponding fibers of \(\ell\mathscr M\).
\begin{proposition}\label{prop:def_j_x}
Fix any \(x\in\X\). Then there exists a unique bounded linear operator
\[
{\sf j}_x={\sf j}_x[\mathscr M,\ell]\colon\ell(\mathscr M^*)_x\hookrightarrow(\ell\mathscr M_x)'
\]
such that the following identity holds:
\[
\langle\,{\sf j}_x(\ell\omega_x),\ell v_x\rangle=\ell\langle\omega,v\rangle(x)\quad\text{ for every }\omega\in\mathscr M^*\text{ and }v\in\mathscr M.
\]
Moreover, we have that \({\sf j}_x\) is an isometric embedding.
\end{proposition}

We do not know whether \({\sf j}_x\) is surjective, cf.\ \cite[Problem 3.11]{DMLP25}. However, in the following result we prove a weaker
statement, namely that the image of \({\sf j}_x\) is weakly\(^*\) dense in the dual of \(\ell\mathscr M_x\).
\begin{lemma}\label{lem:img_lift_dual_wstar_dense}
Fix any \(x\in\X\). Then it holds that \({\sf j}_x(\ell(\mathscr M^*)_x)\) is weakly\(^*\) dense in \((\ell\mathscr M_x)'\).
\end{lemma}
\begin{proof}
For any \(v\in\mathscr M\), we pick an element \(\omega^v\in{\rm Dual}(v)\), which exists by Theorem \ref{thm:Hahn-Banach}. Then
\[
\langle\,{\sf j}_x(\ell\omega^v_x),\ell v_x\rangle=\ell\langle\omega^v,v\rangle(x)=\ell|v|(x)=\|\ell v_x\|_{\ell\mathscr M_x}.
\]
In particular, \({\sf j}_x(\ell(\mathscr M^*)_x)\) separates the points of \(\ell\mathscr M_x\).
Thanks to \eqref{eq:equiv_w*_dense}, the statement follows.
\end{proof}
\begin{remark}\label{rmk:fibers_lift_Linfty(mm)}{\rm
As proved in \cite[Example 4.8]{GLP22} (or as it follows from Remark \ref{rmk:lift_Linfty(mm)}), we have that
\[
\ell L^\infty(\mm)_x\cong\R\quad\text{ for every }x\in\X.
\]
Indeed, identifying \(\ell L^\infty(\mm)\) with \(\mathcal L^\infty(\Sigma)\) as in Remark \ref{rmk:lift_Linfty(mm)},
one can readily check that the map
\[
\ell L^\infty(\mm)_x\ni f_x\mapsto\left\{\begin{array}{ll}
f(x)\\
\ell(\pi_\mm(f))(x)
\end{array}\quad\begin{array}{ll}
\text{ if }x\in\X\setminus{\rm A}(\ell),\\
\text{ if }x\in{\rm A}(\ell)
\end{array}\right.
\]
provides an isomorphism of Banach spaces between \(\ell L^\infty(\mm)_x\) and \(\R\).
\fr}\end{remark}
\subsubsection*{Liftings of bounded \(L^\infty(\mm)\)-multilinear operators}
As we are going to see, it is possible to lift a bounded \(L^\infty(\mm)\)-multilinear operator to a
bounded \(\mathcal L^\infty(\Sigma)\)-multilinear operator. The proof of the following
result is inspired by \cite[Proposition 4.14]{GLP22}, where the particular case \(k=1\) is considered.
\begin{theorem}[Liftings of bounded \(L^\infty(\mm)\)-multilinear operators]\label{thm:lift_bdd_bilin}
Let \(\mathscr M^1,\ldots,\mathscr M^k,\mathscr N\) be \(L^\infty(\mm)\)-Banach \(L^\infty(\mm)\)-modules.
Let \(\Phi\in{\rm M}_k(\mathscr M^1,\ldots,\mathscr M^k;\mathscr N)\) be given. Then there exists a unique operator
\(\ell\Phi\in{\rm M}_k(\ell\mathscr M^1,\ldots,\ell\mathscr M^k;\ell\mathscr N)\) such that
\[\begin{tikzcd}
\mathscr M^1\times\ldots\times\mathscr M^k \arrow[d,swap,"\ell\times\ldots\times\ell"] \arrow[r,"\Phi"] & \mathscr N \arrow[d,"\ell"] \\
\ell\mathscr M^1\times\ldots\times\ell\mathscr M^k \arrow[r,swap,"\ell\Phi"] & \ell\mathscr N
\end{tikzcd}\]
is a commutative diagram, where we set
\[
(\ell\times\ldots\times\ell)(v_1,\ldots,v_k)\coloneqq(\ell v_1,\ldots,\ell v_k)\quad\text{ for every }(v_1,\ldots,v_k)\in\mathscr M^1\times\ldots\times\mathscr M^k.
\]
Moreover, it holds that \(|\ell\Phi|\in\mathcal L^\infty(\Sigma)\) and \(|\ell\Phi|=\ell|\Phi|\).
\end{theorem}
\begin{proof}
For any \(i=1,\ldots,k\), we denote by \(V_i\) the image of \(\ell\colon\mathscr M^i\to\ell\mathscr M^i\). Let us define
\[
(\ell\Phi)(\ell v_1,\ldots,\ell v_k)\coloneqq\ell(\Phi(v_1,\ldots,v_k))\in\ell\mathscr N\quad\text{ for every }(v_1,\ldots,v_k)\in\mathscr M^1\times\ldots\times\mathscr M^k.
\]
The resulting map \(\ell\Phi\colon V_1\times\ldots\times V_k\to\ell\mathscr N\) is well defined, thanks to the following estimate:
\[
|\ell(\Phi(v_1,\ldots,v_k))|=\ell|\Phi(v_1,\ldots,v_k)|\leq\ell(|\Phi||v_1|\ldots|v_k|)=(\ell|\Phi|)|\ell v_1|\ldots|\ell v_k|.
\]
Clearly, the map \(\ell\Phi\colon V_1\times\ldots\times V_k\to\ell\mathscr N\) is multilinear by construction. Note also that
\[
|(\ell\Phi)(\bar v_1,\ldots,\bar v_k)|\leq(\ell|\Phi|)|\bar v_1|\ldots|\bar v_k|\quad\text{ for every }(\bar v_1,\ldots,\bar v_k)\in V_1\times\ldots\times V_k.
\]
Since \(V_i\) generates \(\ell\mathscr M_i\) for every \(i=1,\ldots,k\), we know from Proposition \ref{prop:extension_bdd_multilin} that \(\ell\Phi\) can be
uniquely extended to an operator \(\ell\Phi\in{\rm M}_k(\ell\mathscr M^1,\ldots,\ell\mathscr M^k;\ell\mathscr N)\), which satisfies \(|\ell\Phi|\leq\ell|\Phi|\).
It only remains to check that \(|\ell\Phi|\geq\ell|\Phi|\). To this aim, fix any \(\varepsilon>0\). By the glueing property, we find
\((v_1,\ldots,v_k)\in\mathbb D_{\mathscr M^1}\times\ldots\times\mathbb D_{\mathscr M^k}\) such that \(|\Phi(v_1,\ldots,v_k)|\geq|\Phi|-\varepsilon\) holds
\(\mm\)-a.e.\ on \(\X\), so that
\[
|\ell\Phi|\geq|(\ell\Phi)(\ell v_1,\ldots,\ell v_k)|=|\ell(\Phi(v_1,\ldots,v_k))|=\ell|\Phi(v_1,\ldots,v_k)|\geq\ell|\Phi|-\varepsilon\quad\text{ on }\X.
\]
By the arbitrariness of \(\varepsilon\), we can finally conclude that \(|\ell\Phi|\geq\ell|\Phi|\), so that \(|\ell\Phi|=\ell|\Phi|\in\mathcal L^\infty(\Sigma)\).
\end{proof}

The following result can be easily deduced from Theorem \ref{thm:lift_bdd_bilin}, taking also Remark \ref{rmk:fibers_lift_Linfty(mm)} 
and Lemma \ref{lem:dens_in_fiber_of_lift_mod} into account.
\begin{corollary}\label{cor:trace_lift_bdd_bilin}
Let \(\mathscr M^1,\ldots,\mathscr M^k,\mathscr N\) be \(L^\infty(\mm)\)-Banach \(L^\infty(\mm)\)-modules. Fix any \(x\in\X\). Let \(\Phi\in{\rm M}_k(\mathscr M^1,\ldots,\mathscr M^k;\mathscr N)\)
be given. Then there exists a unique \(\ell\Phi_x\in{\rm M}_k(\ell\mathscr M^1_x,\ldots,\ell\mathscr M^k_x;\ell\mathscr N_x)\) such that
\[
(\ell\Phi_x)(\ell v^1_x,\ldots,\ell v^k_x)=\ell(\Phi(v^1,\ldots,v^k))_x\quad\text{ for every }(v^1,\ldots,v^k)\in\mathscr M^1\times\ldots\times\mathscr M^k.
\]
Moreover, it holds that \(\|\ell\Phi_x\|_{{\rm M}_k(\ell\mathscr M^1_x,\ldots,\ell\mathscr M^k_x;\ell\mathscr N_x)}=\ell|\Phi|(x)\).
In particular, if \(\varphi\in{\rm M}_k(\mathscr M^1,\ldots,\mathscr M^k)\) is given, then there exists a unique  \(\ell\varphi_x\in{\rm M}_k(\ell\mathscr M^1_x,\ldots,\ell\mathscr M^k_x)\) such that
\[
(\ell\varphi_x)(\ell v^1_x,\ldots,\ell v^k_x)=\ell(\varphi(v^1,\ldots,v^k))(x)\quad\text{ for every }(v^1,\ldots,v^k)\in\mathscr M^1\times\ldots\times\mathscr M^k
\]
and it holds that \(\|\ell\varphi_x\|_{{\rm M}_k(\ell\mathscr M^1_x,\ldots,\ell\mathscr M^k_x)}=\ell|\varphi|(x)\).
\end{corollary}
\section{Tensor products of measurable Banach bundles}
\subsection{Injective tensor products}\label{sec:inj_mod}
Throughout this section, the following objects will be fixed:
\begin{itemize}
\item \((\X,\Sigma,\mm)\): a complete probability space.
\item \(\ell\): a lifting of \(\mm\).
\item \(\mathscr M\), \(\mathscr N\): \(L^\infty(\mm)\)-Banach \(L^\infty(\mm)\)-modules.
\item \(\mathbb U\): a universal separable Banach space.
\item \({\sf e}\colon\mathbb U\hat\otimes_\varepsilon\mathbb U\hookrightarrow\mathbb U\): some fixed linear
isometric embedding (which exists by Remark \ref{rmk:tensor_Banach_separable}).
\item \({\bf E}\), \({\bf F}\): separable Banach \(\mathbb U\)-bundles over \((\X,\Sigma)\).
\item \({\sf i}_\star=\{{\sf i}_x\}_{x\in\X}\): the canonical embeddings
\({\sf i}_x\colon{\bf E}(x)\hat\otimes_\varepsilon{\bf F}(x)\hookrightarrow\mathbb U\hat\otimes_\varepsilon\mathbb U\),
cf.\ Remark \ref{rmk:inj_tensor_products_subspaces}.
\item \(\mathcal C\), \(\mathcal D\): fiberwise dense, countable \(\mathbb Q\)-vector subspaces of \(\bar\Gamma({\bf E})\)
and \(\bar\Gamma({\bf F})\), respectively (which exist due to Proposition \ref{prop:select_dense}).
\end{itemize}
\begin{lemma}\label{lem:aux_isometric_emb_injective_lift}
Fix any \(\alpha=\sum_{i=1}^n v^i\otimes w^i\in\mathscr M\otimes_\varepsilon\mathscr N\) and \(x\in\X\). Then it holds that
\begin{equation}\label{eq:aux_isometric_emb_injective_lift}
\|\ell\alpha_x\|_{\ell(\mathscr M\hat\otimes_\varepsilon\mathscr N)_x}
=\sup\bigg\{\sum_{i=1}^n\langle\,{\sf j}_x(\ell\omega_x),\ell v^i_x\rangle
\langle\,{\sf j}_x(\ell\eta_x),\ell w^i_x\rangle\;\bigg|\;
\omega\in\mathbb D_{\mathscr M^*},\,\eta\in\mathbb D_{\mathscr N^*}\bigg\},
\end{equation}
where \({\sf j}_x\coloneqq{\sf j}_x[\mathscr M,\ell]\colon\ell(\mathscr M^*)_x\hookrightarrow(\ell\mathscr M_x)'\)
is given by Proposition \ref{prop:def_j_x}.
\end{lemma}
\begin{proof}
Given any \(k\in\N\), we know from Remark \ref{rmk:equiv_inj_tens_norm} that there exist
\(\omega^k\in\mathbb D_{\mathscr M^*}\) and \(\eta^k\in\mathbb D_{\mathscr N^*}\) such that
\(|\alpha|_\varepsilon\leq\sum_{i=1}^n\langle\omega^k,v^i\rangle\langle\eta^k,w^i\rangle+\frac{1}{k}\).
Denoting by \(R_x\) the right-hand side of \eqref{eq:aux_isometric_emb_injective_lift}, we get
\[\begin{split}
\|\ell\alpha_x\|_{\ell(\mathscr M\hat\otimes_\varepsilon\mathscr N)_x}&=\ell|\alpha|_\varepsilon(x)
\leq\ell\bigg(\sum_{i=1}^n\langle\omega^k,v^i\rangle\langle\eta^k,w^i\rangle\bigg)(x)+\frac{1}{k}\\
&=\sum_{i=1}^n\langle\,{\sf j}_x(\ell\omega^k_x),\ell v^i_x\rangle\langle\,{\sf j}_x(\ell\eta^k_x),\ell w^i_x\rangle+\frac{1}{k}
\leq R_x+\frac{1}{k}.
\end{split}\]
Letting \(k\to\infty\), we obtain that \(\|\ell\alpha_x\|_{\ell(\mathscr M\hat\otimes_\varepsilon\mathscr N)_x}\leq R_x\).
Conversely, \(\sum_{i=1}^n\langle\omega,v^i\rangle\langle\eta,w^i\rangle\leq|\alpha|_\varepsilon\) holds for every
\(\omega\in\mathbb D_{\mathscr M^*}\) and \(\eta\in\mathbb D_{\mathscr N^*}\), thus
\(\sum_{i=1}^n\langle\,{\sf j}_x(\ell\omega_x),\ell v^i_x\rangle\langle\,{\sf j}_x(\ell\eta_x),\ell w^i_x\rangle
\leq\|\ell\alpha_x\|_{\ell(\mathscr M\hat\otimes_\varepsilon\mathscr N)_x}\), whence it follows that
\(R_x\leq\|\ell\alpha_x\|_{\ell(\mathscr M\hat\otimes_\varepsilon\mathscr N)_x}\) by the arbitrariness of \(\omega\)
and \(\eta\). This gives \eqref{eq:aux_isometric_emb_injective_lift}.
\end{proof}
\begin{corollary}\label{cor:isometric_emb_injective_lift}
Fix any \(x\in\X\). Then there exists a unique bounded linear operator
\[
\phi_x=\phi_x[\mathscr M,\mathscr N,\ell]\colon\ell\mathscr M_x\hat\otimes_\varepsilon\ell\mathscr N_x\to\ell(\mathscr M\hat\otimes_\varepsilon\mathscr N)_x
\]
such that \(\phi_x(\ell v_x\otimes\ell w_x)=\ell(v\otimes w)_x\) for all \(v\in\mathscr M\) and \(w\in\mathscr N\).
Moreover, \(\phi_x\) is a linear isometry.
\end{corollary}
\begin{proof}
Note that \(\mathcal V_x\coloneqq\{\ell v_x:v\in\mathscr M\}\otimes\{\ell w_x:w\in\mathscr N\}\) is a vector subspace
of \(\ell\mathscr M_x\hat\otimes_\varepsilon\ell\mathscr N_x\). The requirements in the statement force us to define
\(\phi_x\colon\mathcal V_x\to\ell(\mathscr M\hat\otimes_\varepsilon\mathscr N)_x\) as
\[
\phi_x\bigg(\sum_{i=1}^n\ell v^i_x\otimes\ell w^i_x\bigg)\coloneqq\sum_{i=1}^n\ell(v^i\otimes w^i)_x
\quad\text{ for every }\sum_{i=1}^n v^i\otimes w^i\in\mathscr M\otimes\mathscr N.
\]
Note that \({\sf j}_x(\ell\omega_x)\in\mathbb D_{(\ell\mathscr M_x)'}\) and \({\sf j}_x(\ell\eta_x)\in\mathbb D_{(\ell\mathscr N_x)'}\)
for every \(\omega\in\mathbb D_{\mathscr M^*}\) and \(\eta\in\mathbb D_{\mathscr N^*}\).
Moreover, Lemma \ref{lem:img_lift_dual_wstar_dense} ensures that \(\{{\sf j}_x(\ell\omega_x):\omega\in\mathscr M\}\)
and \(\{{\sf j}_x(\ell\eta_x):\eta\in\mathscr N\}\) are weakly\(^*\) dense in \((\ell\mathscr M_x)'\) and 
\((\ell\mathscr N_x)'\), respectively. Using also Lemma \ref{lem:aux_isometric_emb_injective_lift}, we deduce that
\[\begin{split}
\bigg\|\sum_{i=1}^n\ell(v^i\otimes w^i)_x\bigg\|_{\ell(\mathscr M\hat\otimes_\varepsilon \mathscr N)_x} &=\sup\bigg\{\sum_{i=1}^n\langle\,{\sf j}_x(\ell\omega_x),\ell v^i_x\rangle
\langle\,{\sf j}_x(\ell\eta_x),\ell w^i_x\rangle\;\bigg|\;\omega\in\mathbb D_{\mathscr M^*},\,\eta\in\mathbb D_{\mathscr N^*}\bigg\}\\
&=\bigg\|\sum_{i=1}^n\ell v^i_x\otimes\ell w^i_x\bigg\|_{\ell\mathscr M_x\hat\otimes_\varepsilon\ell\mathscr N_x}.
\end{split}\]
This shows that \(\phi_x\) is well defined, as well as linear and isometric. Since \(\mathcal V_x\) is dense
in \(\ell\mathscr M_x\hat\otimes_\varepsilon\ell\mathscr N_x\) by Lemma \ref{lem:dens_in_fiber_of_lift_mod}
and Remark \ref{rmk:tensor_Banach_separable}, we finally conclude that \(\phi_x\) can be uniquely extended to
a linear isometry \(\phi_x\colon\ell\mathscr M_x\hat\otimes_\varepsilon\ell\mathscr N_x\hookrightarrow\ell(\mathscr M\hat\otimes_\varepsilon\mathscr N)_x\).
\end{proof}
\begin{proposition}
Let \(\phi_x\coloneqq\phi_x[\mathscr M,\mathscr N,\ell]\colon\ell\mathscr M_x\hat\otimes_\varepsilon\ell\mathscr N_x
\to\ell(\mathscr M\hat\otimes_\varepsilon\mathscr N)_x\) be as in Corollary \ref{cor:isometric_emb_injective_lift}.
Let us define the \(\mathcal L^\infty(\Sigma)\)-Banach \(\mathcal L^\infty(\Sigma)\)-submodule
\(\mathscr W=\mathscr W[\mathscr M,\mathscr N,\ell]\) of \(\ell(\mathscr M\hat\otimes_\varepsilon\mathscr N)\) as
\[
\mathscr W\coloneqq\Big\{\alpha\in\ell(\mathscr M\hat\otimes_\varepsilon\mathscr N)\;\Big|\;
\alpha_x\in\phi_x(\ell\mathscr M_x\hat\otimes_\varepsilon\ell\mathscr N_x)\text{ for }\mm\text{-a.e.\ }x\in\X\Big\}.
\]
Then \(\Pi_\mm(\mathscr W)=\Pi_\mm(\ell(\mathscr M\hat\otimes_\varepsilon\mathscr N))\).
In particular, \(\Pi_\mm(\mathscr W)\) is an \(L^\infty(\mm)\)-Banach \(L^\infty(\mm)\)-module and
\[
\mathscr M\hat\otimes_\varepsilon\mathscr N\cong\Pi_\mm(\mathscr W).
\]
\end{proposition}
\begin{proof}
Since \(\phi_x(\ell\mathscr M_x\hat\otimes_\varepsilon\ell\mathscr N_x)\) is a closed vector subspace of
\(\ell(\mathscr M\hat\otimes_\varepsilon\mathscr N)_x\) for every \(x\in\X\), it is easy to verify that \(\mathscr W\) is an
\(\mathcal L^\infty(\Sigma)\)-Banach \(\mathcal L^\infty(\Sigma)\)-submodule of \(\ell(\mathscr M\hat\otimes_\varepsilon\mathscr N)\).
Now, fix \(\alpha\in\ell(\mathscr M\hat\otimes_\varepsilon\mathscr N)\). Given any \(k\in\N\), we can find an element
\(\alpha^k\in\ell(\mathscr M\hat\otimes_\varepsilon\mathscr N)\) of the form \(\alpha^k=\sum_{n\in\N}\1_{E_{k,n}}\cdot\ell\alpha^{k,n}\),
where \((E_{k,n})_{n\in\N}\subseteq\Sigma\) is a partition of \(\X\) and \((\alpha^{k,n})_{n\in\N}\subseteq\mathscr M\otimes_\varepsilon\mathscr N\),
with \(\|\alpha-\alpha^k\|_{\ell(\mathscr M\hat\otimes_\varepsilon\mathscr N)}\leq\frac{1}{k}\).
Since for any \(x\in\X\) we have that
\[
\ell(v\otimes w)_x=\phi_x(\ell v_x\otimes\ell w_x)\in\phi_x(\ell\mathscr M_x\hat\otimes_\varepsilon\ell\mathscr N_x)
\quad\text{ for every }v\in\mathscr M\text{ and }w\in\mathscr N,
\]
we deduce that \(\alpha^{k,n}_x\in\phi_x(\ell\mathscr M_x\hat\otimes_\varepsilon\ell\mathscr N_x)\) for every
\(k,n\in\N\) and \(x\in\X\). Given that \(\ell\alpha^{k,n}_x=\alpha^k_x\) for \(\mm\)-a.e.\ \(x\in E_{k,n}\),
we obtain that \(\alpha^k_x\in\phi_x(\ell\mathscr M_x\hat\otimes_\varepsilon\ell\mathscr N_x)\) for \(\mm\)-a.e.\ \(x\in\X\)
and thus \((\alpha^k)_{k\in\N}\subseteq\mathscr W\). Using the fact that \(\alpha^k_x\to\alpha_x\) in
\(\ell(\mathscr M\hat\otimes_\varepsilon\mathscr N)_x\) for every \(x\in\X\), we can conclude that \(\alpha\in\mathscr W\),
so that \(\pi_\mm(\alpha)\in\Pi_\mm(\mathscr W)\). This shows that \(\Pi_\mm(\mathscr W)=\Pi_\mm(\ell(\mathscr M\hat\otimes_\varepsilon\mathscr N))\), as desired.
Taking \eqref{eq:ell_inv_a.e.} into account, we obtain in particular that \(\Pi_\mm(\mathscr W)\cong\mathscr M\hat\otimes_\varepsilon\mathscr N\). The proof is then complete.
\end{proof}

Hereafter, we focus only on the case \(\mathscr M=\Gamma({\bf E})\) and \(\mathscr N=\Gamma({\bf F})\).
\begin{remark}\label{rmk:measurability}{\rm 
Given any \(v\in\bar\Gamma({\bf E})\) and \(w\in\bar\Gamma({\bf F})\), we claim that 
\[
\X\ni x\mapsto({\sf e}\circ{\sf i}_x)(v(x)\otimes w(x))\in\mathbb U\quad\text{ is a measurable map.}
\]
To prove it, let us show that it can be expressed as a composition of measurable maps. Indeed:
\begin{itemize}
\item \(\X\ni x\mapsto(v(x),w(x))\in\mathbb U\times\mathbb U\) is measurable (as \(v\) and \(w\) are measurable),
\item \(\mathbb U\times\mathbb U\ni(\bar v,\bar w)\mapsto\bar v\otimes\bar w\in\mathbb U\hat\otimes_\varepsilon\mathbb U\)
is continuous,
\item \({\bf E}(x)\otimes_\varepsilon{\bf F}(x)\ni\bar v\otimes\bar w\mapsto{\sf i}_x(\bar v\otimes\bar w)\in\mathbb U\hat\otimes_\varepsilon\mathbb U\)
is the inclusion map,
\item \({\sf e}\colon\mathbb U\hat\otimes_\varepsilon\mathbb U\hookrightarrow\mathbb U\) is isometric.
\end{itemize}
All in all, the claim is proven.
\fr}\end{remark}
\begin{definition}[Injective tensor products of separable Banach bundles]\label{def:injective_tp}
Let us define
\[
({\bf E}\hat\otimes_\varepsilon{\bf F})(x)\coloneqq({\sf e}\circ{\sf i}_x)\big({\bf E}(x)\hat\otimes_\varepsilon{\bf F}(x)\big)
\quad\text{ for every }x\in\X.
\]
We call \({\bf E}\hat\otimes_\varepsilon{\bf F}\colon\X\twoheadrightarrow\mathbb U\) the \textbf{injective tensor product}
of \({\bf E}\) and \({\bf F}\).
\end{definition}

 Let us check that \({\bf E}\hat\otimes_\varepsilon{\bf F}\) is indeed a separable Banach \(\mathbb U\)-bundle.
Letting \(\mathcal F\) be the set of all finite subsets of \(\mathcal C\times\mathcal D\), for any \(F\in\mathcal F\) we define
the map \(\gamma_F\colon\X\to\mathbb U\) as
\begin{equation}\label{eq:def_gamma_F}
\gamma_F(x)\coloneqq\sum_{(v,w)\in F}({\sf e}\circ{\sf i}_x)(v(x)\otimes w(x))\quad\text{ for every }x\in\X.
\end{equation}
Given that \(\{v(x):v\in\mathcal C\}\) and \(\{w(x):w\in\mathcal D\}\) are dense \(\mathbb Q\)-vector subspaces of \({\bf E}(x)\) and \({\bf F}(x)\), respectively,
we know from Remark \ref{rmk:tensor_Banach_separable} that
\begin{equation}\label{eq:prop_gamma_F}
{\rm cl}_{\mathbb U}(\{\gamma_F(x)\;|\;F\in\mathcal F\})=({\bf E}\hat\otimes_\varepsilon{\bf F})(x)\quad\text{ for every }x\in\X.
\end{equation}
Applying Proposition \ref{prop:select_dense}, we conclude that \({\bf E}\hat\otimes_\varepsilon{\bf F}\) is a separable Banach \(\mathbb U\)-bundle, as desired.
\begin{theorem}[The identification \(\Gamma({\bf E}\hat\otimes_\varepsilon{\bf F})\cong\Gamma({\bf E})\hat\otimes_\varepsilon\Gamma({\bf F})\)]\label{thm:inj_prod_bundles}
Under the above assumptions, there exists a unique operator
\(\Psi=\Psi[{\bf E},{\bf F},{\sf i}_\star,{\sf e}]\in\textsc{Hom}(\Gamma({\bf E})\hat\otimes_\varepsilon\Gamma({\bf F});\Gamma({\bf E}\hat\otimes_\varepsilon{\bf F}))\) such that
\[
\Psi\big(\pi_\mm(v)\otimes\pi_\mm(w)\big)=\pi_\mm\big(({\sf e}\circ{\sf i}_\star)(v(\star)\otimes w(\star))\big)
\quad\text{ for every }v\in\bar\Gamma({\bf E})\text{ and }w\in\bar\Gamma({\bf F}).
\]
Moreover, \(\Psi\) is an isomorphism of \(L^\infty(\mm)\)-Banach \(L^\infty(\mm)\)-modules, thus in particular
\[
\Gamma({\bf E})\hat\otimes_\varepsilon\Gamma({\bf F})\cong\Gamma({\bf E}\hat\otimes_\varepsilon{\bf F}).
\]
\end{theorem}
\begin{proof}
First, note that we are forced to define \(\Psi\colon\Gamma({\bf E})\otimes_\varepsilon\Gamma({\bf F})\to\Gamma({\bf E}\hat\otimes_\varepsilon{\bf F})\) as
\[
\Psi\bigg(\sum_{i=1}^n\pi_\mm(v_i)\otimes\pi_\mm(w_i)\bigg)=\sum_{i=1}^n\pi_\mm\big(({\sf e}\circ{\sf i}_\star)(v_i(\star)\otimes w_i(\star))\big)
\]
for every \(n\in\N\), \(v_1,\ldots,v_n\in\bar\Gamma({\bf E})\) and \(w_1,\ldots,w_n\in\bar\Gamma({\bf F})\). Let us now check that \(\Psi\) is well posed.
Thanks to Theorems \ref{thm:dual_section_space} and \ref{thm:Hahn-Banach}, we can find \((\omega_v)_{v\in\mathcal C}\subseteq\bar\Gamma({\bf E}'_{w^*})\) and
\((\eta_w)_{w\in\mathcal D}\subseteq\bar\Gamma({\bf F}'_{w^*})\) such that
\[\begin{split}
\|\omega_v(x)\|_{{\bf E}(x)'}=\1_{\{v(\star)\neq 0\}}(x),&\quad\langle\omega_v(x),v(x)\rangle=\|v(x)\|_{{\bf E}(x)}\quad\text{ for }\mm\text{-a.e.\ }x\in\X,\\
\|\eta_w(x)\|_{{\bf F}(x)'}=\1_{\{w(\star)\neq 0\}}(x),&\quad\langle\eta_w(x),w(x)\rangle=\|w(x)\|_{{\bf F}(x)}\quad\text{ for }\mm\text{-a.e.\ }x\in\X
\end{split}\]
for any given \(v\in\mathcal C\) and \(w\in\mathcal D\). Using the fact that \({\sf e}\circ{\sf i}_x\) is isometric and Lemma \ref{lem:separating_points}, we obtain
\[\begin{split}
\bigg\|\sum_{i=1}^n({\sf e}\circ{\sf i}_x)(v_i(x)\otimes w_i(x))\bigg\|_{({\bf E}\hat\otimes_\varepsilon{\bf F})(x)}
&=\bigg\|\sum_{i=1}^n v_i(x)\otimes w_i(x)\bigg\|_{{\bf E}(x)\hat\otimes_\varepsilon{\bf F}(x)}\\
&=\sup\bigg\{\sum_{i=1}^n\langle\omega_v(x),v_i(x)\rangle\langle\eta_w(x),w_i(x)\rangle\;\bigg|\;v\in\mathcal C,\,w\in\mathcal D\bigg\}
\end{split}\]
for \(\mm\)-a.e.\ \(x\in\X\), whence (using again Theorem \ref{thm:dual_section_space} and Lemma \ref{lem:separating_points}) it follows that
\[\begin{split}
\bigg|\sum_{i=1}^n\pi_\mm\big(({\sf e}\circ{\sf i}_\star)(v_i(\star)\otimes w_i(\star))\big)\bigg|
&=\bigvee\bigg\{\pi_\mm\bigg(\sum_{i=1}^n\langle\omega_v(\star),v_i(\star)\rangle\langle\eta_w(\star),w_i(\star)\rangle\bigg)\;\bigg|\;v\in\mathcal C,\,w\in\mathcal D\bigg\}\\
&=\bigg|\sum_{i=1}^n\pi_\mm(v_i)\otimes\pi_\mm(w_i)\bigg|_\varepsilon.
\end{split}\]
Therefore, the map \(\Psi\colon\Gamma({\bf E})\otimes_\varepsilon\Gamma({\bf F})\to\Gamma({\bf E}\hat\otimes_\varepsilon{\bf F})\) is well posed. Moreover, \(\Psi\) is linear
and satisfies \(|\Psi(\alpha)|=|\alpha|_\varepsilon\) for every \(\alpha\in\Gamma({\bf E})\otimes_\varepsilon\Gamma({\bf F})\), thus it can be uniquely extended to an operator
\(\Psi\in\textsc{Hom}(\Gamma({\bf E})\hat\otimes_\varepsilon\Gamma({\bf F});\Gamma({\bf E}\hat\otimes_\varepsilon{\bf F}))\) satisfying \(|\Psi(\alpha)|=|\alpha|_\varepsilon\) for every
\(\alpha\in\Gamma({\bf E})\hat\otimes_\varepsilon\Gamma({\bf F})\). To conclude, it remains to show that \(\Psi\colon\Gamma({\bf E})\hat\otimes_\varepsilon\Gamma({\bf F})\to\Gamma({\bf E}\hat\otimes_\varepsilon{\bf F})\)
is surjective. To this aim, let us take \(\{\gamma_F:F\in\mathcal F\}\subseteq\bar\Gamma({\bf E}\hat\otimes_\varepsilon{\bf F})\) as in \eqref{eq:def_gamma_F}.
As a consequence of \eqref{eq:prop_gamma_F} and Proposition \ref{prop:select_dense}, we have that \(\{\pi_\mm(\gamma_F):F\in\mathcal F\}\) generates \(\Gamma({\bf E}\hat\otimes_\varepsilon{\bf F})\).
Since \(\pi_\mm(\gamma_F)=\Psi\big(\sum_{(v,w)\in F}\pi_\mm(v)\otimes\pi_\mm(w)\big)\), we have that \(\{\pi_\mm(\gamma_F):F\in\mathcal F\}\subseteq\Psi(\Gamma({\bf E})\hat\otimes_\varepsilon\Gamma({\bf F}))\),
thus we finally conclude that \(\Psi\) is surjective.
\end{proof}
\subsection{Projective tensor products}\label{sec:proj_mod}
Throughout this section, the following objects will be fixed:
\begin{itemize}
\item \((\X,\Sigma,\mm)\): a complete probability space.
\item \(\ell\): a lifting of \(\mm\).
\item \(\mathscr M\), \(\mathscr N\): \(L^\infty(\mm)\)-Banach \(L^\infty(\mm)\)-modules.
\item \(\mathbb U\): a universal separable Banach space.
\item \({\bf E}\), \({\bf F}\): separable Banach \(\mathbb U\)-bundles over \((\X,\Sigma)\).
\item \(\mathcal C\), \(\mathcal D\): fiberwise dense, countable \(\mathbb Q\)-vector subspaces of \(\bar\Gamma({\bf E})\)
and \(\bar\Gamma({\bf F})\), respectively.
\end{itemize}
\begin{definition}\label{def:ell_pi^x}
Fix any \(x\in\X\).
Then we define \(\ell^x_\pi=\ell^x_\pi[\mathscr M,\mathscr N,\ell]\colon\mathscr M\otimes_\pi\mathscr N\to\ell\mathscr M_x\otimes_\pi\ell\mathscr N_x\) as
\[
\ell^x_\pi\alpha\coloneqq\sum_{i=1}^n\ell v^i_x\otimes\ell w^i_x\quad\text{ for every }\alpha=\sum_{i=1}^n v^i\otimes w^i\in\mathscr M\otimes_\pi\mathscr N.
\]
\end{definition}

Let us check that \(\ell^x_\pi\) is well defined. Fix any tensor
\(\sum_{i=1}^n v^i\otimes w^i=\sum_{j=1}^m\tilde v^j\otimes\tilde w^j\in\mathscr M\otimes_\pi\mathscr N\).
Let \({\sf j}_x^{\mathscr M}\coloneqq{\sf j}_x[\mathscr M,\ell]\) and \({\sf j}_x^{\mathscr N}\coloneqq{\sf j}_x[\mathscr N,\ell]\)
be given by Proposition \ref{prop:def_j_x}. Then for any \(\omega\in\mathscr M^*\) and \(\eta\in\mathscr N^*\) we have that
\[\begin{split}
\sum_{i=1}^n\langle\,{\sf j}_x^{\mathscr M}(\ell\omega_x),\ell v^i_x\rangle\langle\,{\sf j}_x^{\mathscr N}(\ell\eta_x),\ell w^i_x\rangle
&=\ell\left(\sum_{i=1}^n\langle\omega,v^i\rangle\langle\eta,w^i\rangle\right)(x)=\ell\left(\sum_{j=1}^m\langle\omega,\tilde v^j\rangle\langle\eta,\tilde w^j\rangle\right)(x)\\
&=\sum_{j=1}^m\langle\,{\sf j}_x^{\mathscr M}(\ell\omega_x),\ell\tilde v^j_x\rangle\langle\,{\sf j}_x^{\mathscr N}(\ell\eta_x),\ell\tilde w^j_x\rangle.
\end{split}\]
Taking Lemma \ref{lem:img_lift_dual_wstar_dense} and \ref{lem:null_tensor_Banach} into account,
we conclude that \(\sum_{i=1}^n\ell v^i_x\otimes\ell w^i_x=\sum_{j=1}^m\ell\tilde v^j_x\otimes\ell\tilde w^j_x\).
\begin{proposition}\label{prop:key_fibers_proj_tensor}
Fix any \(x\in\X\). Then there exists a unique linear \(1\)-Lipschitz operator
\[
\iota_x=\iota_x[\mathscr M,\mathscr N,\ell]\colon\ell\mathscr M_x\hat\otimes_\pi\ell\mathscr N_x\to\ell(\mathscr M\hat\otimes_\pi\mathscr N)_x
\]
such that \(\iota_x(\ell v_x\otimes\ell w_x)=\ell(v\otimes w)_x\) for every \(v\in\mathscr M\) and \(w\in\mathscr N\).
\end{proposition}
\begin{proof}
Define \(V_x\coloneqq\{\ell v_x:v\in\mathscr M\}\) and \(W_x\coloneqq\{\ell w_x:w\in\mathscr N\}\). Recall
from Lemma \ref{lem:dens_in_fiber_of_lift_mod} that \(V_x\) and \(W_x\) are dense vector subspaces of
\(\ell\mathscr M_x\) and \(\ell\mathscr N_x\), respectively, thus in particular \(V_x\otimes W_x\) is a
dense vector subspace of \(\ell\mathscr M_x\hat\otimes_\pi\ell\mathscr N_x\). Letting
\(\ell^x_\pi\coloneqq\ell^x_\pi[\mathscr M,\mathscr N,\ell]\) be as in Definition \ref{def:ell_pi^x},
we point out that \(\ell^x_\pi(\mathscr M\otimes\mathscr N)=V_x\otimes W_x\).
Next, let us define the map \(\iota_x\colon V_x\otimes W_x\to\ell(\mathscr M\hat\otimes_\pi\mathscr N)_x\) as
\[
\iota_x(\ell^x_\pi\alpha)\coloneqq\sum_{i=1}^n\ell(v^i\otimes w^i)_x\quad\text{ for every }\alpha=\sum_{i=1}^n v^i\otimes w^i\in\mathscr M\otimes\mathscr N.
\]
We need to check that \(\iota_x\) is well defined. To this aim, fix \(\alpha=\sum_{i=1}^n v^i\otimes w^i\in\mathscr M\otimes\mathscr N\). Combining Theorems
\ref{thm:dual_proj_tens_prod} and \ref{thm:Hahn-Banach}, we can find an element \(b^\alpha\in{\rm B}(\mathscr M,\mathscr N)\) such that \(|b^\alpha|\leq 1\)
and \(\tilde b^\alpha_\pi(\alpha)=|\alpha|_\pi\). Moreover, observe that if an element \(b\in{\rm B}(\mathscr M,\mathscr N)\) with \(|b|\leq 1\) is given,
then for any \(x\in\X\) the element \(\ell b_x\in{\rm B}(\ell\mathscr M_x,\ell\mathscr N_x)\) provided by Corollary \ref{cor:trace_lift_bdd_bilin} satisfies
\(\|\ell b_x\|_{{\rm B}(\ell\mathscr M_x,\ell\mathscr N_x)}\leq 1\). Therefore,
\[\begin{split}
\left\|\sum_{i=1}^n\ell(v^i\otimes w^i)_x\right\|_{\ell(\mathscr M\hat\otimes_\pi\mathscr N)_x}&=\left|\sum_{i=1}^n\ell(v^i\otimes w^i)\right|(x)=\ell|\alpha|_\pi(x)
=\ell(\tilde b^\alpha_\pi(\alpha))(x)\\
&\leq\sup_{b\in\mathbb D_{{\rm B}(\mathscr M,\mathscr N)}}\ell(\tilde b_\pi(\alpha))(x)=\sup_{b\in\mathbb D_{{\rm B}(\mathscr M,\mathscr N)}}\ell\left(\sum_{i=1}^n b(v^i,w^i)\right)(x)\\
&=\sup_{b\in\mathbb D_{{\rm B}(\mathscr M,\mathscr N)}}\sum_{i=1}^n(\ell b_x)(\ell v^i_x,\ell w^i_x)\leq\|\ell^x_\pi\alpha\|_{\ell\mathscr M_x\hat\otimes_\pi\ell\mathscr N_x},
\end{split}\]
thanks to the last statement of Theorem \ref{thm:dual_proj_tens_prod}. These estimates show that \(\iota_x\) is well defined and
\[
\|\iota_x(\bar\alpha)\|_{\ell(\mathscr M\hat\otimes_\pi\mathscr N)_x}\leq\|\bar\alpha\|_{\ell\mathscr M_x\hat\otimes_\pi\ell\mathscr N_x}\quad\text{ for every }\bar\alpha\in V_x\otimes W_x.
\]
Since \(\iota_x\colon V_x\otimes W_x\to\ell(\mathscr M\hat\otimes_\pi\mathscr N)_x\) is linear by construction and \(V_x\otimes W_x\) is dense in
\(\ell\mathscr M_x\hat\otimes_\pi\ell\mathscr N_x\), it follows that \(\iota_x\colon V_x\otimes W_x\to\ell(\mathscr M\hat\otimes_\pi\mathscr N)_x\)
can be uniquely extended to a linear \(1\)-Lipschitz operator \(\iota_x\colon\ell\mathscr M_x\hat\otimes_\pi\ell\mathscr N_x\to\ell(\mathscr M\hat\otimes_\pi\mathscr N)_x\).
Consequently, the statement is achieved.
\end{proof}

By unwrapping the definitions, one can easily see that
\begin{equation}\label{eq:formula_iota_x_ell_pi_x}
\iota_x(\ell_\pi^x\alpha)=\ell\alpha_x\quad\text{ for every }\alpha\in\mathscr M\otimes_\pi\mathscr N\text{ and }x\in\X. 
\end{equation}
\begin{lemma}\label{lem:key_fibers_proj_tensor_isom_ae}
Fix any \(\alpha\in\mathscr M\otimes_\pi\mathscr N\). Then there exists a set
\({\rm N}_\alpha={\rm N}_\alpha[\mathscr M,\mathscr N,\ell]\in\mathcal N_\mm\) such that
\[
\|\ell\alpha_x\|_{\ell(\mathscr M\hat\otimes_\pi\mathscr N)_x}=\|\ell^x_\pi\alpha\|_{\ell\mathscr M_x\hat\otimes_\pi\ell\mathscr N_x}\quad\text{ for every }x\in\X\setminus{\rm N}_\alpha,
\]
where \(\ell^x_\pi\coloneqq\ell^x_\pi[\mathscr M,\mathscr N,\ell]\) and \(\iota_x\coloneqq\iota_x[\mathscr M,\mathscr N,\ell]\)
are given by Definition \ref{def:ell_pi^x} and Proposition \ref{prop:key_fibers_proj_tensor}, respectively.
\end{lemma}
\begin{proof}
For any \(k\in\N\), we can find \((n^k_j)_{j\in\N}\subseteq\N\), \((v^{k,j,i})_{i=1}^{n^k_j}\subseteq\mathscr M\) and \((w^{k,j,i})_{i=1}^{n^k_j}\subseteq\mathscr N\) for \(j\in\N\),
and a partition \((E^k_j)_{j\in\N}\subseteq\Sigma\) of \(\X\) such that \(\alpha=\sum_{i=1}^{n^k_j}v^{k,j,i}\otimes w^{k,j,i}\) for every \(j\in\N\) and
\[
\1_{E^k_j}^\mm\sum_{i=1}^{n^k_j}|v^{k,j,i}||w^{k,j,i}|\leq\1_{E^k_j}^\mm|\alpha|_\pi+\frac{1}{k}\quad\text{ for every }j\in\N.
\]
Therefore, for any \(j\in\N\) and \(x\in\ell E^k_j\) we have that
\[\begin{split}
\|\ell^x_\pi\alpha\|_{\ell\mathscr M_x\hat\otimes_\pi\ell\mathscr N_x}&=\left\|\sum_{i=1}^{n^k_j}\ell v^{k,j,i}_x\otimes\ell w^{k,j,i}_x\right\|_{\ell\mathscr M_x\hat\otimes_\pi\ell\mathscr N_x}
\leq\sum_{i=1}^{n^k_j}\|\ell v^{k,j,i}_x\|_{\ell\mathscr M_x}\|\ell w^{k,j,i}_x\|_{\ell\mathscr N_x}\\
&=\ell\left(\sum_{i=1}^{n^k_j}|v^{k,j,i}||w^{k,j,i}|\right)(x)\leq\ell|\alpha|_\pi(x)+\frac{1}{k}=\|\ell\alpha_x\|_{\ell(\mathscr M\hat\otimes_\pi\mathscr N)_x}+\frac{1}{k}\\
&=\|\iota_x(\ell^x_\pi\alpha)\|_{\ell(\mathscr M\hat\otimes_\pi\mathscr N)_x}+\frac{1}{k}.
\end{split}\]
Letting \({\rm N}_\alpha\) be the \(\mm\)-null set \(\bigcup_{k\in\N}\bigcap_{j\in\N}\X\setminus\ell E^k_j\) and recalling \eqref{eq:formula_iota_x_ell_pi_x},
we get the statement.
\end{proof}

Hereafter, we focus only on the case \(\mathscr M=\Gamma({\bf E})\) and \(\mathscr N=\Gamma({\bf F})\).
We use the shorthand notation
\[
\ell{\bf E}_x\coloneqq\ell\Gamma({\bf E})_x\quad\text{ for every }x\in\X
\]
and similarly \(\ell{\bf F}_x\coloneqq\ell\Gamma({\bf F})_x\) for every \(x\in\X\).
\begin{lemma}\label{lem:fibers_bundle_embed_lift}
There exists a set \(\tilde{\rm N}_{\bf E}=\tilde{\rm N}_{\bf E}[\ell,\mathcal C]\in\mathcal N_\mm\) such that the following
property holds: given any \(x\in\X\setminus\tilde{\rm N}_{\bf E}\), there exists a unique linear isometry
\(\psi^{\bf E}_x=\psi^{\bf E}_x[\ell,\mathcal C]\colon{\bf E}(x)\to\ell{\bf E}_x\) such that
\begin{equation}\label{eq:def_psi_x}
\psi^{\bf E}_x(v(x))=\ell(\pi_\mm(v))_x\quad\text{ for every }v\in\mathcal C.
\end{equation}
\end{lemma}
\begin{proof}
We define the \(\mm\)-negligible set \(\tilde{\rm N}_{\bf E}\in\Sigma\) as
\[
\tilde{\rm N}_{\bf E}\coloneqq\bigcup_{v\in\mathcal C}\big\{x\in\X\;\big|\;\ell|\pi_\mm(v)|(x)\neq|v|(x)\big\}.
\]
Now, fix any \(x\in\X\setminus\tilde{\rm N}_{\bf E}\). Given any \(v\in\mathcal C\), we have that
\[
\|\ell(\pi_\mm(v))_x\|_{\ell{\bf E}_x}=\ell|\pi_\mm(v)|(x)=|v|(x)=\|v(x)\|_{{\bf E}(x)},
\]
which guarantees that the map \(\psi^{\bf E}_x\colon\{v(x):v\in\mathcal C\}\to\ell{\bf E}_x\) defined
as in \eqref{eq:def_psi_x} is well posed. It also follows that \(\psi^{\bf E}_x\) is a \(\mathbb Q\)-linear
isometry. Since \(\{v(x):v\in\mathcal C\}\) is a dense \(\mathbb Q\)-vector subspace of \({\bf E}(x)\), we conclude
that \(\psi^{\bf E}_x\) can be uniquely extended to a linear isometry \(\psi^{\bf E}_x\colon{\bf E}(x)\to\ell{\bf E}_x\).
\end{proof}
\begin{remark}\label{rmk:E_and_ellE}{\rm
Lemma \ref{lem:fibers_bundle_embed_lift} gives that \({\bf E}(x)\) is (isomorphic to) a subspace of \(\ell{\bf E}_x\) for
\(\mm\)-a.e.\ \(x\in\X\). However, one should not expect \({\bf E}(x)\cong\ell{\bf E}_x\), since -- in general --
the embeddings \(\psi^{\bf E}_x\colon{\bf E}(x)\hookrightarrow\ell{\bf E}_x\) can fail to be surjective for \emph{every} point
\(x\in\X\setminus\tilde{\rm N}_{\bf E}\), as we shall see in Proposition \ref{prop:psi_x_not_surj}.
\fr}\end{remark}

Our next goal is to define the projective tensor product \({\bf E}\hat\otimes_\pi{\bf F}\) of \({\bf E}\) and \({\bf F}\).
To this aim, we first need to introduce several auxiliary objects:
\begin{itemize}
\item We define the index set \(\mathcal I=\mathcal I[\mathcal C,\mathcal D]\) as the collection of all finite subsets of
\[
\big\{(\pi_\mm(v),\pi_\mm(w))\;\big|\;v\in\mathcal C,\,w\in\mathcal D\big\}.
\]
Note that \(\mathcal I\) is a countable family whose elements are finite subsets of \(\Gamma({\bf E})\times\Gamma({\bf F})\).
\item Given any element \(G\in\mathcal I\), we define the tensor \(\alpha^G\) as
\[
\alpha^G\coloneqq\sum_{(v,w)\in G}v\otimes w\in\Gamma({\bf E})\otimes_\pi\Gamma({\bf F}).
\]
\item We define the set \({\rm N}={\rm N}[\ell,\mathcal C,\mathcal D]\in\mathcal N_\mm\) as
\[
{\rm N}\coloneqq\tilde{\rm N}_{\bf E}\cup\tilde{\rm N}_{\bf F}\cup\bigcup_{G\in\mathcal I}{\rm N}_{\alpha^G},
\]
where \({\rm N}_{\alpha^G}\coloneqq{\rm N}_{\alpha^G}[\Gamma({\bf E}),\Gamma({\bf F}),\ell]\) is given by Lemma
\ref{lem:key_fibers_proj_tensor_isom_ae}, while the sets \(\tilde{\rm N}_{\bf E}\coloneqq\tilde{\rm N}_{\bf E}[\ell,\mathcal C]\)
and \(\tilde{\rm N}_{\bf F}\coloneqq\tilde{\rm N}_{\bf F}[\ell,\mathcal D]\) are given by Lemma \ref{lem:fibers_bundle_embed_lift}.
\item Letting \(\ell_\pi^x\coloneqq\ell_\pi^x[\Gamma({\bf E}),\Gamma({\bf F}),\ell]\) be as in Definition \ref{def:ell_pi^x}, we define
\[
\Lambda_x=\Lambda_x[\ell,\mathcal C,\mathcal D]\coloneqq
\left\{\begin{array}{ll}
{\rm cl}_{\ell{\bf E}_x\hat\otimes_\pi\ell{\bf F}_x}(\{\ell_\pi^x\alpha^G\;|\;G\in\mathcal I\}),\\
\{0\}
\end{array}\quad\begin{array}{ll}
\text{ for every }x\in\X\setminus{\rm N},\\
\text{ for every }x\in{\rm N}.
\end{array}\right.
\]
Since \(\{\ell_\pi^x\alpha^G:G\in\mathcal I\}\) is a \(\mathbb Q\)-vector subspace of \(\ell{\bf E}_x\hat\otimes_\pi\ell{\bf F}_x\), we deduce that
\(\Lambda_x\) is a closed vector subspace of \(\ell{\bf E}_x\hat\otimes_\pi\ell{\bf F}_x\), thus in particular it is a separable Banach space.
\item Given any \(G\in\mathcal I\), we define the element \(\beta_G\in\prod_{x\in\X}\Lambda_x\) as
\[
\beta_G(x)\coloneqq\left\{\begin{array}{ll}
\ell_\pi^x\alpha^G\\
0
\end{array}\quad\begin{array}{ll}
\text{ for every }x\in\X\setminus{\rm N},\\
\text{ for every }x\in{\rm N}.
\end{array}\right.
\]
\item Given any \(G\in\mathcal I\), we fix an element \(\omega^G\in{\rm Dual}(\alpha^G)\subseteq(\Gamma({\bf E})\hat\otimes_\pi\Gamma({\bf F}))^*\),
whose existence follows from Theorem \ref{thm:Hahn-Banach}. Fix \(x\in\X\setminus{\rm N}\). Taking
\(\iota_x\coloneqq\iota_x[\Gamma({\bf E}),\Gamma({\bf F}),\ell]\) as in Proposition \ref{prop:key_fibers_proj_tensor}
and \({\sf j}_x\coloneqq{\sf j}_x[\Gamma({\bf E})\hat\otimes_\pi\Gamma({\bf F}),\ell]\) as in Proposition \ref{prop:def_j_x},
we define \(\theta_G(x)\in\Lambda'_x\) as
\[
\langle\theta_G(x),\bar\alpha\rangle\coloneqq\langle\,{\sf j}_x(\ell\omega^G_x),\iota_x(\bar\alpha)\rangle
\quad\text{ for every }\bar\alpha\in\Lambda_x.
\]
Finally, we define \(\theta_G(x)\coloneqq 0\in\Lambda'_x\) for every \(x\in{\rm N}\).
\end{itemize}
\begin{lemma}\label{lem:frakE_bundle_aux}
Under the above assumptions, let us define
\[
\mathfrak E=\mathfrak E[\ell,\mathcal C,\mathcal D]\coloneqq
\big(\{\Lambda_x\}_{x\in\X},\{\beta_G\}_{G\in\mathcal I},\{\theta_G\}_{G\in\mathcal I}\big).
\]
Then \(\mathfrak E\) is a measurable collection of separable Banach spaces.
\end{lemma}
\begin{proof}
First, observe that \(\Lambda_x={\rm cl}_{\Lambda_x}(\{\beta_G(x):G\in\mathcal I\})\) for every \(x\in\X\)
by construction. Moreover, given any \(H,G\in\mathcal I\) and \(x\in\X\), we can compute
\begin{equation}\label{eq:frakE_formula_aux}\begin{split}
\langle\theta_H(x),\beta_G(x)\rangle&=\1_{\X\setminus{\rm N}}(x)\langle\,{\sf j}_x(\ell\omega^H_x),\iota_x(\ell_\pi^x\alpha^G)\rangle
=\1_{\X\setminus{\rm N}}(x)\sum_{(v,w)\in G}\langle\,{\sf j}_x(\ell\omega^H_x),\iota_x(\ell v_x\otimes\ell w_x)\rangle\\
&=\1_{\X\setminus{\rm N}}(x)\sum_{(v,w)\in G}\langle\,{\sf j}_x(\ell\omega^H_x),\ell(v\otimes w)_x\rangle
=\1_{\X\setminus{\rm N}}(x)\sum_{(v,w)\in G}\ell\langle\omega^H,v\otimes w\rangle(x)\\
&=\1_{\X\setminus{\rm N}}(x)\,\ell\langle\omega^H,\alpha^G\rangle(x),
\end{split}\end{equation}
thus in particular \(\X\ni x\mapsto\langle\theta_H(x),\beta_G(x)\rangle\in\R\) is measurable. Finally, if
\(x\in{\rm N}\) then we have \(0=\|\beta_G(x)\|_{\Lambda_x}=\|\theta_G(x)\|_{\Lambda'_x}=\langle\theta_G(x),\beta_G(x)\rangle\),
while if \(x\in\X\setminus{\rm N}\) then \eqref{eq:frakE_formula_aux} and Lemma \ref{lem:key_fibers_proj_tensor_isom_ae} give
\[\begin{split}
\langle\theta_G(x),\beta_G(x)\rangle&=\ell\langle\omega^G,\alpha^G\rangle(x)=\ell|\alpha^G|_\pi(x)=\|\ell\alpha^G_x\|_{\ell(\Gamma({\bf E})\hat\otimes_\pi\Gamma({\bf F}))_x}
=\|\ell_\pi^x\alpha^G\|_{\ell{\bf E}_x\hat\otimes_\pi\ell{\bf F}_x}\\
&=\|\beta_G(x)\|_{\Lambda_x},\\
\|\theta_G(x)\|_{\Lambda'_x}&=\|\,{\sf j}_x(\ell\omega^G_x)\circ\iota_x\|_{\Lambda'_x}
\leq\|\,{\sf j}_x(\ell\omega^G_x)\|_{(\ell(\Gamma({\bf E})\hat\otimes_\pi\Gamma({\bf F}))_x)'}
=\|\ell\omega^G_x\|_{\ell((\Gamma({\bf E})\hat\otimes_\pi\Gamma({\bf F}))^*)_x}\\
&=\ell|\omega^G|(x)=\ell\1_{\{|\alpha^G|_\pi>0\}}^\mm(x)=\1_{\ell\{|\alpha^G|_\pi>0\}}(x)
=\1_{\{\ell\alpha^G_\star\neq 0\}}(x)=\1_{\{\beta_G(\star)\neq 0\}}(x).
\end{split}\]
All in all, we have shown that \(\mathfrak E\) is a measurable collection of separable Banach spaces.
\end{proof}

In view of Lemma \ref{lem:frakE_bundle_aux}, we can now give the following definition.
\begin{definition}[Projective tensor products of separable Banach bundles]
Let \(\mathfrak E\coloneqq\mathfrak E[\ell,\mathcal C,\mathcal D]\) be as in Lemma \ref{lem:frakE_bundle_aux}.
Fix a measurable collection of linear isometric embeddings \({\rm I}_\star=\{{\rm I}_x\}_{x\in\X}\) associated
with \(\mathfrak E\), whose existence is guaranteed by Theorem \ref{thm:embed_meas_coll}. Then we define
\[
{\bf E}\hat\otimes_\pi{\bf F}\colon\X\twoheadrightarrow\mathbb U
\]
as the separable Banach \(\mathbb U\)-bundle induced by \({\rm I}_\star\). We call \({\bf E}\hat\otimes_\pi{\bf F}\)
the \textbf{projective tensor product} of \({\bf E}\) and \({\bf F}\).
\end{definition}
\begin{remark}[Identification of the fibers of \({\bf E}\hat\otimes_\pi{\bf F}\)]\label{rmk:fibers_proj_bundle}{\rm
Under the above assumptions, we claim that
\begin{equation}\label{eq:char_fibers_proj_bundle}
\Lambda_x={\rm cl}_{\ell{\bf E}_x\hat\otimes_\pi\ell{\bf F}_x}\big(\psi_x^{\bf E}({\bf E}(x))\otimes\psi_x^{\bf F}({\bf F}(x))\big)
\quad\text{ for every }x\in\X\setminus{\rm N},
\end{equation}
where \(\psi_x^{\bf E}\coloneqq\psi_x^{\bf E}[\ell,\mathcal C]\colon{\bf E}(x)\to\ell{\bf E}_x\) and
\(\psi_x^{\bf F}\coloneqq\psi_x^{\bf F}[\ell,\mathcal D]\colon{\bf F}(x)\to\ell{\bf F}_x\) are given by Lemma
\ref{lem:fibers_bundle_embed_lift}. Indeed, for any finite subset \(\bar G\) of \(\mathcal C\times\mathcal D\)
we have that \(G\coloneqq\{(\pi_\mm(v),\pi_\mm(w)):(v,w)\in\bar G\}\in\mathcal I\) and
\[
\beta_G(x)=\ell_\pi^x\alpha^G=\sum_{(v,w)\in\bar G}\ell(\pi_\mm(v))_x\otimes\ell(\pi_\mm(w))_x
=\sum_{(v,w)\in\bar G}\psi_x^{\bf E}(v(x))\otimes\psi_x^{\bf F}(w(x)).
\]
Since \(\mathcal C\) and \(\mathcal D\) are fiberwise dense in \(\bar\Gamma({\bf E})\) and \(\bar\Gamma({\bf F})\),
respectively, and \(\psi_x^{\bf E}\), \(\psi_x^{\bf F}\) are linear isometries by Lemma \ref{lem:fibers_bundle_embed_lift},
we deduce that \(\{\psi_x^{\bf E}(v(x)):v\in\mathcal C\}\) and \(\{\psi_x^{\bf F}(w(x)):w\in\mathcal D\}\) are dense
\(\mathbb Q\)-vector subspaces of \(\psi_x^{\bf E}({\bf E}(x))\) and \(\psi_x^{\bf F}({\bf F}(x))\), respectively.
Recalling Remark \ref{rmk:tensor_Banach_separable}, we thus obtain \eqref{eq:char_fibers_proj_bundle}.
Since \(({\bf E}\hat\otimes_\pi{\bf F})(x)={\rm I}_x(\Lambda_x)\cong\Lambda_x\) for every \(x\in\X\) by Theorem
\ref{thm:embed_meas_coll}, it follows from \eqref{eq:char_fibers_proj_bundle} that
\begin{equation}\label{eq:char_fibers_proj_bundle_final}
({\bf E}\hat\otimes_\pi{\bf F})(x)\cong{\rm cl}_{\ell{\bf E}_x\hat\otimes_\pi\ell{\bf F}_x}(\psi_x^{\bf E}\big({\bf E}(x))\otimes\psi_x^{\bf F}({\bf F}(x))\big)\quad\text{ for }\mm\text{-a.e.\ }x\in\X.
\end{equation}
In light of Remark \ref{rmk:proj_tensor_products_subspaces} and Proposition \ref{prop:psi_x_not_surj}, it might happen that
\(({\bf E}\hat\otimes_\pi{\bf F})(x)\ncong{\bf E}(x)\hat\otimes_\pi{\bf F}(x)\) for every \(x\) belonging to some set
of positive \(\mm\)-measure. However, sufficient conditions ensuring that \(({\bf E}\hat\otimes_\pi{\bf F})(x)\cong{\bf E}(x)\hat\otimes_\pi{\bf F}(x)\)
for \(\mm\)-a.e.\ \(x\) will be given in Theorem \ref{thm:suff_ident_proj_fibers} and Corollary \ref{cor:proj_Leb-Boch}.
\fr}\end{remark}
\begin{theorem}[The identification \(\Gamma({\bf E}\hat\otimes_\pi{\bf F})\cong\Gamma({\bf E})\hat\otimes_\pi\Gamma({\bf F})\)]
\label{thm:proj_mod}
Under the above assumptions, there exists a unique bounded \(L^\infty(\mm)\)-bilinear operator
\(\mathfrak p\colon\Gamma({\bf E})\times\Gamma({\bf F})\to\Gamma({\bf E}\hat\otimes_\pi{\bf F})\) such that
\begin{equation}\label{eq:def_frakp}
\mathfrak p(\pi_\mm(v),\pi_\mm(w))=\pi_\mm\big({\rm I}_\star(\beta_{v,w}(\star))\big)\quad\text{ for every }v\in\mathcal C\text{ and }w\in\mathcal D.
\end{equation}
where we set \(\beta_{v,w}\coloneqq\beta_{\{(\pi_\mm(v),\pi_\mm(w))\}}\) for brevity. Moreover, it holds that
\begin{equation}\label{eq:char_frakp}
(\Gamma({\bf E}\hat\otimes_\pi{\bf F}),\mathfrak p)\cong(\Gamma({\bf E})\hat\otimes_\pi\Gamma({\bf F}),\otimes).
\end{equation}
\end{theorem}
\begin{proof}
First, we define the map \(\mathfrak p\colon\{\pi_\mm(v):v\in\mathcal C\}\times\{\pi_\mm(w):w\in\mathcal D\}\to\Gamma({\bf E}\hat\otimes_\pi{\bf F})\) as in
\eqref{eq:def_frakp}. Given any \((v,w)\in\mathcal C\times\mathcal D\) and \(x\in\X\setminus{\rm N}\), we deduce from Lemma \ref{lem:key_fibers_proj_tensor_isom_ae} that
\[\begin{split}
\|{\rm I}_x(\beta_{v,w}(x))\|_{({\bf E}\hat\otimes_\pi{\bf F})(x)}&=\|\beta_{v,w}(x)\|_{\Lambda_x}
=\big\|\ell(\pi_\mm(v))_x\otimes\ell(\pi_\mm(w))_x\big\|_{\ell{\bf E}_x\hat\otimes_\pi\ell{\bf F}_x}\\
&=\big\|\ell(\pi_\mm(v)\otimes\pi_\mm(w))_x\big\|_{\ell(\Gamma({\bf E})\hat\otimes_\pi\Gamma({\bf F}))_x}
=\ell|\pi_\mm(v)\otimes\pi_\mm(w)|_\pi(x)\\
&=\ell\big(|\pi_\mm(v)||\pi_\mm(w)|\big)(x)=\ell|\pi_\mm(v)|(x)\,\ell|\pi_\mm(w)|(x),
\end{split}\]
whence it follows that
\[
\big|\pi_\mm\big({\rm I}_\star(\beta_{v,w}(\star))\big)\big|=|\pi_\mm(v)||\pi_\mm(w)|\quad\text{ for every }(v,w)\in\mathcal C\times\mathcal D.
\]
This shows that the map \(\mathfrak p\) is well defined. The fact that \(\mathcal C\) and \(\mathcal D\) are fiberwise dense \(\mathbb Q\)-vector subspaces
of \(\bar\Gamma({\bf E})\) and \(\bar\Gamma({\bf F})\), respectively, implies that \(\{\pi_\mm(v):v\in\mathcal C\}\) and \(\{\pi_\mm(w):w\in\mathcal D\}\)
are generating \(\mathbb Q\)-vector subspaces of \(\Gamma({\bf E})\) and \(\Gamma({\bf F})\), respectively. Since we have that \(\mathfrak p\) is
\(\mathbb Q\)-bilinear by construction, we know from Proposition \ref{prop:extension_bdd_multilin} that \(\mathfrak p\) can be uniquely extended to an operator
\[
\mathfrak p\in{\rm B}(\Gamma({\bf E}),\Gamma({\bf F});\Gamma({\bf E}\hat\otimes_\pi{\bf F})).
\]
Let us now pass to the verification of \eqref{eq:char_frakp}. Our goal is to show that the pair \((\Gamma({\bf E}\hat\otimes_\pi{\bf F}),\mathfrak p)\) satisfies the universal
property that is stated in Theorem \ref{thm:proj_tens_Ban_mod}. To this aim, fix an \(L^\infty(\mm)\)-Banach \(L^\infty(\mm)\)-module \(\mathscr Q\) and an operator
\(b\in{\rm B}(\Gamma({\bf E}),\Gamma({\bf F});\mathscr Q)\). We claim that there exists a unique operator \(T\in\textsc{Hom}(\Gamma({\bf E}\hat\otimes_\pi{\bf F});\mathscr Q)\)
such that the diagram
\begin{equation}\label{eq:diagram_for_T}\begin{tikzcd}
\Gamma({\bf E})\times\Gamma({\bf F}) \arrow[d,swap,"\mathfrak p"] \arrow[r,"b"] & \mathscr Q \\
\Gamma({\bf E}\hat\otimes_\pi{\bf F}) \arrow[ur,swap,"T"]
\end{tikzcd}\end{equation}
commutes. Observe that, given any \(G\in\mathcal I\) and letting \(\bar\beta_G\coloneqq{\rm I}_\star(\beta_G(\star))\in\bar\Gamma({\bf E}\hat\otimes_\pi{\bf F})\) for brevity,
we are forced to define
\begin{equation}\label{eq:def_univ_morph_T}
T(\pi_\mm(\bar\beta_G))\coloneqq\sum_{(v,w)\in G}b(\pi_\mm(v),\pi_\mm(w)).
\end{equation}
Let us check that the above definition is well posed. For any \(x\in\X\setminus{\rm N}\), Corollary \ref{cor:trace_lift_bdd_bilin} yields
\[\begin{split}
\bigg\|\ell\bigg(\sum_{(v,w)\in G}b(\pi_\mm(v),\pi_\mm(w))\bigg)_x\bigg\|_{\ell\mathscr Q_x}&=\bigg\|\sum_{(v,w)\in G}\ell\big(b(\pi_\mm(v),\pi_\mm(w))\big)_x\bigg\|_{\ell\mathscr Q_x}\\
&=\bigg\|\sum_{(v,w)\in G}(\ell b_x)\big(\ell(\pi_\mm(v))_x,\ell(\pi_\mm(w))_x\big)\bigg\|_{\ell\mathscr Q_x}\\
&=\bigg\|\sum_{(v,w)\in G}\widetilde{(\ell b_x)}_\pi\big(\ell(\pi_\mm(v))_x\otimes\ell(\pi_\mm(w))_x\big)\bigg\|_{\ell\mathscr Q_x}\\
&=\big\|\widetilde{(\ell b_x)}_\pi(\ell_\pi^x\alpha^G)\big\|_{\ell\mathscr Q_x}\\
&\leq\big\|\widetilde{(\ell b_x)}_\pi\big\|_{\textsc{Hom}(\ell{\bf E}_x\hat\otimes_\pi\ell{\bf F}_x;\ell\mathscr Q_x)}\|\ell_\pi^x\alpha^G\|_{\ell{\bf E}_x\hat\otimes_\pi\ell{\bf F}_x}\\
&=\|\ell b_x\|_{{\rm B}(\ell{\bf E}_x,\ell{\bf F}_x;\ell\mathscr Q_x)}\|\beta_G(x)\|_{\Lambda_x}=\ell|b|(x)\,\|\bar\beta_G(x)\|_{({\bf E}\hat\otimes_\pi{\bf F})(x)}.
\end{split}\]
It follows that
\[
\bigg|\sum_{(v,w)\in G}b(\pi_\mm(v),\pi_\mm(w))\bigg|\leq|b||\pi_\mm(\bar\beta_G)|\quad\text{ for every }G\in\mathcal I,
\]
thus in particular the definition in \eqref{eq:def_univ_morph_T} is well posed. Note that \(\{\pi_\mm(\bar\beta_G):G\in\mathcal I\}\) is a \(\mathbb Q\)-vector subspace
of \(\Gamma({\bf E}\hat\otimes_\pi{\bf F})\) and the map \(T\colon\{\pi_\mm(\bar\beta_G):G\in\mathcal I\}\to\mathscr Q\) is \(\mathbb Q\)-linear by construction.
Since it can also be easily shown that \(\{\pi_\mm(\bar\beta_G):G\in\mathcal I\}\) generates \(\Gamma({\bf E}\hat\otimes_\pi{\bf F})\), we know from Proposition
\ref{prop:extension_bdd_multilin} that \(T\) can be uniquely extended to an operator \(T\in\textsc{Hom}(\Gamma({\bf E}\hat\otimes_\pi{\bf F});\mathscr Q)\).
Clearly, \(T\) is the unique element of \(\textsc{Hom}(\Gamma({\bf E}\hat\otimes_\pi{\bf F});\mathscr Q)\) for which \eqref{eq:diagram_for_T} is a commutative diagram.
Due to the arbitrariness of \(\mathscr Q\) and \(b\), it follows that \eqref{eq:char_frakp} holds, thus the statement is finally achieved.
\end{proof}
\begin{theorem}\label{thm:suff_ident_proj_fibers}
Assume \({\bf E}(x)\) and \({\bf F}(x)\) are Hilbert spaces for every \(x\in\X\). Then it holds that
\[
({\bf E}\hat\otimes_\pi{\bf F})(x)\cong{\bf E}(x)\hat\otimes_\pi{\bf F}(x)\quad\text{ for }\mm\text{-a.e.\ }x\in\X.
\]
\end{theorem}
\begin{proof}
Since \({\bf E}(x)\) and \({\bf F}(x)\) are Hilbert spaces for every \(x\in\X\), we have that the
\(L^\infty(\mm)\)-Banach \(L^\infty(\mm)\)-modules \(\Gamma({\bf E})\) and \(\Gamma({\bf F})\) are Hilbertian
(see e.g.\ \cite[Theorem 3.1]{LPV22}). Taking Remark \ref{rmk:Hilbertian_lifting} into account, we deduce
that \(\ell{\bf E}_x\) and \(\ell{\bf F}_x\) are Hilbert spaces for every \(x\in\X\). Therefore,
\eqref{eq:char_fibers_proj_bundle_final} and the last part of Remark \ref{rmk:proj_tensor_products_subspaces}
guarantee that
\[
({\bf E}\hat\otimes_\pi{\bf F})(x)\cong
{\rm cl}_{\ell{\bf E}_x\hat\otimes_\pi\ell{\bf F}_x}\big(\psi_x^{\bf E}({\bf E}(x))\otimes\psi_x^{\bf F}({\bf F}(x))\big)
\cong{\bf E}(x)\hat\otimes_\pi{\bf F}(x)\quad\text{ for }\mm\text{-a.e.\ }x\in\X,
\]
thus proving the statement.
\end{proof}

Let us now focus on the special case of constant bundles. Namely, we additionally assume that
\begin{equation}\label{eq:hp_const_bundles}
{\bf E}(x)=\B,\qquad{\bf F}(x)=\mathbb V\quad\text{ for every }x\in\X,
\end{equation}
for some separable Banach spaces \(\B\) and \(\mathbb V\). In particular, \(\Gamma({\bf E})=L^\infty(\mm;\B)\)
and \(\Gamma({\bf F})=L^\infty(\mm;\mathbb V)\).
\begin{proposition}\label{prop:psi_x_not_surj}
Assume \eqref{eq:hp_const_bundles} holds and \(\mathcal C=\{\1_\X\bar v:\bar v\in Q\}\) for
some \(\mathbb Q\)-vector subspace \(Q\) of \(\B\). Let \(\tilde{\rm N}_{\bf E}=\tilde{\rm N}_{\bf E}[\ell,\mathcal C]\)
and \(\big\{\psi_x^{\bf E}=\psi_x^{\bf E}[\ell,\mathcal C]\big\}_{x\in\X\setminus\tilde{\rm N}_{\bf E}}\) be as in Lemma
\ref{lem:fibers_bundle_embed_lift}. If the measure \(\mm\) is atomless and the Banach space \(\B\) is infinite dimensional,
then it holds that
\[
\psi_x^{\bf E}(\B)\neq\ell\B_x\quad\text{ for every }x\in\X\setminus\tilde{\rm N}_{\bf E},
\]
where we denote \(\ell\B_x\coloneqq\ell{\bf E}_x\) for every \(x\in\X\).
\end{proposition}
\begin{proof}
The next argument is inspired by the proof of \cite[Theorem C.2]{LP23} and by results of \cite{Gutman93}.
Fix \(x\in\X\setminus\tilde{\rm N}_{\bf E}\). We argue by contradiction: suppose
\(\psi_x^{\bf E}(\B)=\ell\B_x\). Fix a sequence \(\{\bar v_n\}_{n\in\N}\subseteq Q\) such that
\(\|\bar v_n\|_{\B}\leq 1\) and \(\|\bar v_n-\bar v_m\|_{\B}\geq\frac{1}{2}\) for every \(n,m\in\N\) with
\(n\neq m\). Denote \(\bar w_n\coloneqq\psi_x^{\bf E}(\bar v_n)\in\ell\B_x\) for every \(n\in\N\). Using Lemma
\ref{lem:point_missing_union_liftings}, we can find a sequence \(\{E_n\}_{n\in\N}\subseteq\Sigma\) of pairwise disjoint
sets with \(\ell E_n=E_n\) for every \(n\in\N\) such that \(\mm\big(\X\setminus\bigcup_{n\in\N}E_n\big)=0\) and
\(x\notin\bigcup_{n\in\N}E_n\). Define
\[
v\coloneqq\sum_{n\in\N}\1_{E_n}^\mm\bar v_n\in\Gamma({\bf E}),\qquad\bar w\coloneqq\ell v_x\in\ell\B_x.
\]
Since we assumed that \(\psi_x^{\bf E}(\B)=\ell\B_x\), there exists \(\bar v\in\B\) such that
\(\psi_x^{\bf E}(\bar v)=\bar w\). Now, take a sequence \(\{\bar u_k\}_{k\in\N}\subseteq Q\) such that
\(\bar u_k\to\bar v\) in \(\B\). It clearly follows that \(\1_\X^\mm\bar u_k\to\1_\X^\mm\bar v\) in \(L^\infty(\mm;\B)\),
so that \(\ell(\1_\X^\mm\bar u_k)\to\ell(\1_\X^\mm\bar v)\) in \(\ell L^\infty(\mm;\B)\) and thus accordingly
\[
\ell(\1_\X^\mm\bar v)_x=\lim_{k\to\infty}\ell(\1_\X^\mm\bar u_k)_x=\lim_{k\to\infty}\psi_x^{\bf E}(\bar u_k)
=\psi_x^{\bf E}(\bar v)=\bar w.
\]
Next, observe that \(\{\bar w_n:n\in\N\}\) is a discrete subset of \(\ell\B_x\), thus either
\(\bar w=\bar w_{n_0}\) for some \(n_0\in\N\), or there exists \(\varepsilon>0\) such that
\(\|\bar w-\bar w_n\|_{\ell\B_x}\geq\varepsilon\) for every \(n\in\N\). Let us show that the second
condition cannot occur. If \(\|\bar v-\bar v_n\|_{\B}=\|\bar w-\bar w_n\|_{\ell\B_x}\geq\varepsilon\)
for every \(n\in\N\), then we have that
\[
0=\|\bar w-\ell(\1_\X^\mm\bar v)_x\|_{\ell\B_x}=\|\ell(v-\1_\X^\mm\bar v)_x\|_{\ell\B_x}
=\ell\bigg(\sum_{n\in\N}\1_{E_n}^\mm\|\bar v_n-\bar v\|_\B\bigg)(x)\geq\ell(\varepsilon\1_\X^\mm)(x)=\varepsilon,
\]
which is impossible. Hence, there exists \(n_0\in\N\) such that \(\bar w=\bar w_{n_0}\). Since
\(x\in\X\setminus E_{n_0}=\ell(\X\setminus E_{n_0})\), we can finally conclude that
\[\begin{split}
0&=\|\bar w-\ell v_x\|_{\ell\B_x}=\|\bar w_{n_0}-\ell v_x\|_{\ell\B_x}
=\big\|\1_{\ell(\X\setminus E_{n_0})}(x)\,\ell(\1_\X^\mm\bar v_{n_0}-v)_x\big\|_{\ell\B_x}\\
&=\ell\big|\1_{\X\setminus E_{n_0}}^\mm\bar v_{n_0}-\1_{\X\setminus E_{n_0}}^\mm\cdot v\big|(x)
=\ell\bigg(\sum_{n\neq n_0}\1_{E_n}^\mm\|\bar v_{n_0}-\bar v_n\|_\B\bigg)(x)
\geq\ell\bigg(\frac{1}{2}\1_{\X\setminus E_{n_0}}^\mm\bigg)(x)=\frac{1}{2},
\end{split}\]
which leads to a contradiction. Therefore, we have proved that \(\psi_x^{\bf E}(\B)\neq\ell\B_x\).
\end{proof}

Nonetheless, even if the fibers \(\ell\B_x\) and \(\ell\mathbb V_x\) are typically non-isomorphic to \(\B\)
and \(\mathbb V\), respectively, the following result still holds.
\begin{lemma}\label{lem:proj_prod_const_bundles}
Assume \eqref{eq:hp_const_bundles}. Then \(\psi_x^{\bf E}(\B)\hat\otimes_\pi\psi_x^{\bf F}(\mathbb V)\)
is a subspace of \(\ell\B_x\hat\otimes_\pi\ell\mathbb V_x\) for every \(x\in\X\).
\end{lemma}
\begin{proof}
Fix \(x\in\X\). The goal is to show that every bounded bilinear form \(b\colon\psi_x^{\bf E}(\B)\times\psi_x^{\bf F}(\mathbb V)\to\R\)
can be extended to a bounded bilinear form \(\bar b\colon\ell\B_x\times\ell\mathbb V_x\to\R\) having the same operator
norm, whence the statement follows thanks to Remark \ref{rmk:proj_tensor_products_subspaces}. Given two simple maps
\(v=\sum_{i\in\N}\1_{E_i}^\mm\bar v_i\in L^\infty(\mm;\B)\) and \(w=\sum_{j\in\N}\1_{F_j}^\mm\bar w_j\in L^\infty(\mm;\mathbb V)\),
meaning that \((E_i)_{i\in\N},(F_j)_{j\in\N}\subseteq\Sigma\) are two partitions of \(\X\) and
\((v_i)_{i\in\N}\subseteq\B\), \((w_j)_{j\in\N}\subseteq\mathbb V\) are two bounded sequences,
we define the function \(\tilde b(v,w)\) as
\[
\tilde b(v,w)\coloneqq\sum_{i,j\in\N}b(\psi_x^{\bf E}(\bar v_i),\psi_x^{\bf F}(\bar w_j))\1_{E_i\cap F_j}^\mm\in L^\infty(\mm).
\]
Given that simple maps are dense in \(L^\infty(\mm;\B)\) and \(L^\infty(\mm;\mathbb V)\), one can check
that the above-defined map \((v,w)\mapsto\tilde b(v,w)\in L^\infty(\mm)\) uniquely extends to an operator
\(\tilde b\in{\rm B}(L^\infty(\mm;\B),L^\infty(\mm;\mathbb V))\) that satisfies
\(|\tilde b|=\|b\|_{{\rm B}(\psi_x^{\bf E}(\B),\psi_x^{\bf F}(\mathbb V))}\1_\X^\mm\). By Corollary
\ref{cor:trace_lift_bdd_bilin}, we have that \(\bar b\coloneqq\ell\tilde b_x\in{\rm B}(\ell\B_x,\ell\mathbb V_x)\) satisfies
\[
\|\bar b\|_{{\rm B}(\ell\B_x,\ell\mathbb V_x)}=\ell|\tilde b|(x)=\|b\|_{{\rm B}(\psi_x^{\bf E}(\B),\psi_x^{\bf F}(\mathbb V))}.
\]
Finally, observe that for any \(\bar v\in\B\) and \(\bar w\in\mathbb V\) we have that
\[\begin{split}
\bar b(\psi_x^{\bf E}(\bar v),\psi_x^{\bf F}(\bar w))&=\ell\tilde b_x\big(\ell(\1_\X^\mm\bar v)_x,\ell(\1_\X^\mm\bar w)_x\big)
=\ell(\tilde b(\1_\X^\mm\bar v,\1_\X^\mm\bar w))(x)=\ell\big(b(\psi_x^{\bf E}(\bar v),\psi_x^{\bf F}(\bar w))\1_\X^\mm\big)(x)\\
&=b(\psi_x^{\bf E}(\bar v),\psi_x^{\bf F}(\bar w)).
\end{split}\]
This shows that \(\bar b\) is an extension of \(b\), whence the statement follows.
\end{proof}

As a consequence of Lemma \ref{lem:proj_prod_const_bundles}, we obtain the following result on projective
tensor products (in the sense of \(L^\infty(\mm)\)-Banach \(L^\infty(\mm)\)-modules) of Lebesgue--Bochner spaces
of exponent \(\infty\).
\begin{corollary}\label{cor:proj_Leb-Boch}
Let \((\X,\Sigma,\mm)\) be a complete probability space and \(\B\), \(\mathbb V\) separable Banach spaces.
Then there exists a unique bounded \(L^\infty(\mm)\)-bilinear map
\(\mathfrak p\colon L^\infty(\mm;\B)\times L^\infty(\mm;\mathbb V)\to L^\infty(\mm;\B\hat\otimes_\pi\mathbb V)\)
such that for any \(v\in L^\infty(\mm;\B)\) and \(w\in L^\infty(\mm;\mathbb V)\) we have that
\[
\mathfrak p(v,w)(x)=v(x)\otimes w(x)\quad\text{ for }\mm\text{-a.e.\ }x\in\X.
\]
Moreover, it holds that
\[
(L^\infty(\mm;\B\hat\otimes_\pi\mathbb V),\mathfrak p)\cong\big(L^\infty(\mm;\B)\hat\otimes_\pi L^\infty(\mm;\mathbb V),\otimes\big).
\]
\end{corollary}
\begin{proof}
It follows from Theorem \ref{thm:proj_mod}, by taking into account the fact that
\(({\bf E}\hat\otimes_\pi{\bf F})(x)\cong\B\hat\otimes_\pi\mathbb V\) holds for \(\mm\)-a.e.\ \(x\in\X\)
by Lemma \ref{lem:proj_prod_const_bundles} and Remark \ref{rmk:fibers_proj_bundle}, where \({\bf E}\) and
\({\bf F}\) are as in \eqref{eq:hp_const_bundles}.
\end{proof}
\subsection{Some comments on Hilbert--Schmidt tensor products}
We conclude the paper with a brief discussion on Hilbert--Schmidt tensor products of Hilbertian separable Banach
\(\mathbb U\)-bundles and of their section spaces. To avoid overloading the presentation further, we only state
results, without proofs; their verification can be obtained by adapting the arguments in Section \ref{sec:inj_mod}.
\medskip

Let \((\X,\Sigma,\mm)\) be a complete probability space and \(\ell\) a lifting of \(\mm\). Let \(\mathscr H\),
\(\mathscr K\) be Hilbertian \(L^\infty(\mm)\)-Banach \(L^\infty(\mm)\)-modules. Following \cite[Section 1.5]{Gigli14},
we define the \textbf{Hilbert--Schmidt tensor product} \(\mathscr H\hat\otimes_{\rm HS}\mathscr K\) of \(\mathscr H\)
and \(\mathscr K\) as the module completion of the \(L^\infty(\mm)\)-normed \(L^\infty(\mm)\)-module
\(\mathscr H\otimes_{\rm HS}\mathscr K=(\mathscr H\otimes\mathscr K,|\cdot|_{\rm HS})\), where the pointwise norm
\(|\cdot|_{\rm HS}\) is defined as
\[
|\alpha|_{\rm HS}\coloneqq\sqrt{\sum_{i,j=1}^n(v_i\cdot v_j)(w_i\cdot w_j)}
\quad\text{ for every }\alpha=\sum_{i=1}^n v_i\otimes w_i\in\mathscr H\otimes\mathscr K.
\]
It holds that the space \(\mathscr H\hat\otimes_{\rm HS}\mathscr K\) is a Hilbertian \(L^\infty(\mm)\)-Banach \(L^\infty(\mm)\)-module.
\medskip

Given any point \(x\in\X\), the map sending \(\ell v_x\otimes\ell w_x\) (with \(v\in\mathscr H\) and
\(w\in\mathscr K\)) to \(\ell(v\otimes w)_x\) can be uniquely extended to a linear isometric embedding
\(\ell\mathscr H_x\hat\otimes_{\rm HS}\ell\mathscr K_x\hookrightarrow\ell(\mathscr H\hat\otimes_{\rm HS}\mathscr K)_x\).
Identifying each \(\ell\mathscr H_x\hat\otimes_{\rm HS}\ell\mathscr K_x\) with its image in \(\ell(\mathscr H\hat\otimes_{\rm HS}\mathscr K)_x\)
under the above embedding, we have that
\[
\mathscr Z\coloneqq\Big\{\alpha\in\ell(\mathscr H\hat\otimes_{\rm HS}\mathscr K)\;\Big|\;
\alpha_x\in\ell\mathscr H_x\hat\otimes_{\rm HS}\ell\mathscr K_x\text{ for }\mm\text{-a.e.\ }x\in\X\Big\}.
\]
is an \(\mathcal L^\infty(\Sigma)\)-Banach \(\mathcal L^\infty(\Sigma)\)-submodule of
\(\ell(\mathscr H\hat\otimes_{\rm HS}\mathscr K)\) and that \(\mathscr H\hat\otimes_{\rm HS}\mathscr K\cong\Pi_\mm(\mathscr Z)\).
\medskip

Now, assume that \({\bf H}\) and \({\bf K}\) are separable Banach \(\mathbb U\)-bundles over \((\X,\Sigma)\),
for some fixed universal separable Banach space \(\mathbb U\), such that \({\bf H}(x)\) and \({\bf K}(x)\) are
Hilbert spaces for all \(x\in\X\). Then the section spaces \(\Gamma({\bf H})\) and \(\Gamma({\bf K})\) are
Hilbertian \(L^\infty(\mm)\)-Banach \(L^\infty(\mm)\)-modules \cite[Theorem 3.1]{LPV22}. For any \(x\in\X\), we have a canonical
embedding \({\sf i}_x\colon{\bf H}(x)\hat\otimes_{\rm HS}{\bf K}(x)\hookrightarrow\mathbb U\hat\otimes_{\rm HS}\mathbb U\).
Fix also some linear isometry \({\sf e}\colon\mathbb U\hat\otimes_{\rm HS}\mathbb U\hookrightarrow\mathbb U\), whose
existence is guaranteed by the separability of \(\mathbb U\hat\otimes_{\rm HS}\mathbb U\). Then we define the
\textbf{Hilbert--Schmidt tensor product} \({\bf H}\hat\otimes_{\rm HS}{\bf K}\colon\X\twoheadrightarrow\mathbb U\)
of \({\bf H}\) and \({\bf K}\) as
\[
({\bf H}\hat\otimes_{\rm HS}{\bf K})(x)\coloneqq({\bf e}\circ{\sf i}_x)\big({\bf H}(x)\hat\otimes_{\rm HS}{\bf K}(x)\big)
\quad\text{ for every }x\in\X.
\]
Then \({\bf H}\hat\otimes_{\rm HS}{\bf K}\) is a separable Banach \(\mathbb U\)-bundle. Moreover, it holds that
\[
\Gamma({\bf H})\hat\otimes_{\rm HS}\Gamma({\bf K})\cong\Gamma({\bf H}\hat\otimes_{\rm HS}{\bf K}).
\]
An isomorphism of \(L^\infty(\mm)\)-Banach \(L^\infty(\mm)\)-modules between \(\Gamma({\bf H})\hat\otimes_{\rm HS}\Gamma({\bf K})\)
and \(\Gamma({\bf H}\hat\otimes_{\rm HS}{\bf K})\) is given by the unique operator
\(\Upsilon\in\textsc{Hom}(\Gamma({\bf H})\hat\otimes_{\rm HS}\Gamma({\bf K});\Gamma({\bf H}\hat\otimes_{\rm HS}{\bf K}))\)
satisfying the identity
\[
\Upsilon\big(\pi_\mm(v)\otimes\pi_\mm(w)\big)=\pi_\mm\big(({\sf e}\circ{\sf i}_\star)(v(\star)\otimes w(\star))\big)
\]
for every \(v\in\bar\Gamma({\bf H})\) and \(w\in\bar\Gamma({\bf K})\).
\def\cprime{$'$} \def\cprime{$'$}

\end{document}